\documentclass[12pt]{amsart}

\title{Model Equivalences}

\author{Michael Benedikt and Ehud Hrushovski}

 \usepackage{paralist}
 
\usepackage{graphicx}
\usepackage{amsmath}
\usepackage{xspace}
\usepackage{url}
\usepackage{amsxtra}

\usepackage{enumitem}
\setlist[enumerate]{leftmargin=*}
  
\usepackage{amscd}
\usepackage{amsthm}
\usepackage{hyperref}
\usepackage{xcolor}

\usepackage[T1]{fontenc} 

\usepackage{amsfonts}
\usepackage{amssymb}
\usepackage{eucal}
\usepackage[all]{xypic}
\usepackage{color}

\newtheorem{thm}{Theorem}[section]
\newtheorem{cor}[thm]{Corollary}
\newtheorem{lem}[thm]{Lemma}
\newtheorem{prop}[thm]{Proposition}

\theoremstyle{definition}
\newtheorem{defn}[thm]{Definition}

\newtheorem{rem}[thm]{Remark}
\newtheorem{question}[thm]{Question}
\newtheorem{example}[thm]{Example}
\newtheorem{warning}[thm]{Warning}
\newtheorem{conjecture}[thm]{Conjecture}

\newtheorem{claim}[thm]{Claim}

\usepackage{yhmath}
\DeclareSymbolFont{largesymbols}{OMX}{yhex}{m}{n}
\DeclareMathAccent{\widetilde}{\mathord}{largesymbols}{"65}

\newcommand{\lequalslprime}{(L,=,L')}
\newcommand{\lequalslprimea}[1]{(L,=_{#1},L')}
\newcommand{\lpluslprime}{(L,+,L')}
\newcommand{\lpluslprimea}[1]{(L,+_{#1},L')}

\newcommand{\myeat}[1]{}
\newcommand{\sorts}{{\mathcal S}}
\newcommand{\coupled}{0}
\newcommand{\modeleq}{\ee}
\newcommand{\bigp}{definable \, groupoid}
\newcommand{\mer}{  \kw{MER}}
\newcommand{\bimer}{{\rm definable ~ \mer}}
\newcommand{\inftybimer}{{ \mer}}
\newcommand{\kw}[1]{{\mathsf{#1}}\xspace}
\newcommand{\thmref}[1]{Theorem~\ref{#1}}
\newcommand{\secref}[1]{\S\ref{#1}}
\newcommand{\lemref}[1]{Lemma~\ref{#1}}

\newcommand{\defref}[1]{Definition~\ref{#1}}
\newcommand{\propref}[1]{Proposition~\ref{#1}}
\newcommand{\corref}[1]{Corollary~\ref{#1}}

\newcommand{\remref}[1]{Remark~\ref{#1}}

\newcommand{\maxred}[1]{L^{#1}}
\newcommand{\maxreddl}[1]{L^{#1}_{DL}}
\newcommand{\bitheorymax}{T \times_{\maxred{\groupoid}} T} 
\newcommand{\groupoid}{\mathcal{G}}

\newcommand{\nc}{\newcommand}
\nc{\renc}{\renewcommand}
\nc{\ssec}{\subsection}
\nc{\sssec}{\subsubsection} 
\nc\ol{\overline}
\nc\wT{\widetilde{T}}
\nc\wh{\widehat}

\nc{\Aa}{{\mathbb{A}}}
\nc{\Bb}{{\mathbb{B}}}
 \nc{\Gg}{{\mathbb{G}}}  

\def\rR{\mathcal{R}}
\renc{\L}{\mathcal{L}}
\nc{\Hh}{{\mathbb{H}}}
 \nc{\Nn}{{\mathbb{N}}}
\nc{\Pp}{{\mathbb{P}}}
\nc{\Rr}{{\mathbb{R}}}
\nc{\BV}{{\mathbb{V}}}
\nc{\BW}{{\mathbb{W}}}
\nc{\Zz}{{\mathbb{Z}}}
\nc{\Qq}{{\mathbb{Q}}}
\nc{\Ss}{{\mathbb{S}}}
\nc{\Cc}{{\mathbb{C}}}
\nc{\Ff}{{\mathbb{F}}}
\nc{\ff}{{\bar{f}}}
 \nc{\EL}{{L_\infty}}
 
\nc{\CA}{{\mathcal{A}}}
\nc{\CB}{{\mathcal{B}}}
 
\nc{\CF}{{\mathcal{F}}}

 \nc{\Las}{\mathsf{Las}}

\nc{\CG}{{\mathcal{G}}}

 \nc{\R}{{\mathsf{Def}}}
 
\nc{\CC}{{\mathcal{C}}}
\nc{\CM}{{\mathsf{M}}}

\nc{\CN}{{\mathcal{N}}}
\nc{\Oog}{{\mathbb{O}}}
\nc{\Oo}{{\mathcal{O}}}
 
\nc{\CQ}{{\mathcal{Q}}} 
\nc{\CS}{{\mathcal{S}}}
\nc{\CT}{{\mathcal{T}}}

\nc{\CV}{{\mathcal{V}}}
 
\nc{\CW}{{\mathcal{W}}}
\nc{\CZ}{{\mathcal{Z}}}

\nc{\modec}{{\overset{e.c.}\models}{}}
\nc{\subsetec}{{\underset{e.c.}\subset}{}}

\nc{\bvo}{{\underset{\omega}\bigvee}}
\nc{\bwo}{{\underset{\omega}\bigwedge}}
 
\nc{\bb}{{\mathbf{b}}}
\nc{\bc}{{\mathbf{c}}}
\nc{\bd}{\partial}
\nc{\be}{{\mathbf{e}}}
\nc{\bj}{{\mathbf{j}}}
\nc{\bn}{{\mathbf{n}}}
\nc{\bp}{{\mathbf{p}}}
\nc{\bq}{{\mathbf{q}}}
\nc{\bF}{{\mathbf{F}}}
\nc{\bu}{{\mathbf{u}}}
\nc{\bv}{{\mathbf{v}}}
\nc{\bx}{{\mathbf{x}}}
\nc{\bs}{{\mathbf{s}}}
\nc{\by}{{\bar{y}}}
\nc{\bw}{{\mathbf{w}}}
\nc{\bA}{{\mathbf{A}}}
\nc{\bK}{{\mathbf{K}}}
\nc{\bI}{{\mathbf{I}}}
\nc{\bB}{{\mathbf{B}}}
\nc{\bG}{{\mathbf{G}}}
\nc{\bC}{{\mathbf{C}}}
\nc{\bD}{{\mathbf{D}}}
\nc{\bP}{{\mathbf{P}}}
\nc{\bH}{{\mathbf{H}}}
\nc{\bM}{{\mathbf{M}}}
\nc{\bN}{{\mathbf{N}}}
\nc{\bV}{{\mathbf{V}}}
\nc{\bU}{{\mathbf{U}}}
\nc{\bL}{{\mathbf{L}}}
\nc{\bT}{{\mathbf{T}}}
\nc{\bW}{{\mathbf{W}}}
\nc{\bX}{{\mathbf{X}}}
\nc{\bY}{{\mathbf{Y}}}
\nc{\bZ}{{\mathbf{Z}}}
\nc{\bS}{{\mathbf{S}}}
\nc{\bSi}{{\bar{\Sigma}}}
\nc{\sA}{{\mathsf{A}}}
\nc{\sB}{{\mathsf{B}}}
\nc{\sC}{{\mathsf{C}}}
\nc{\sD}{{\mathsf{D}}}
\nc{\sF}{{\mathsf{F}}}
\nc{\sG}{{\mathsf{G}}}
\nc{\sK}{{\mathsf{K}}}
\nc{\sM}{{\mathsf{M}}}
\nc{\sO}{{\mathsf{O}}}
\nc{\sQ}{{\mathsf{Q}}}
\nc{\sP}{{\mathsf{P}}}
\nc{\sZ}{{\mathsf{Z}}}
\nc{\sfp}{{\mathsf{p}}}
\nc{\sr}{{\mathsf{r}}}
\nc{\sg}{{\mathsf{g}}}
\nc{\sff}{{\mathsf{f}}}
\nc{\sfb}{{\mathsf{b}}}
\nc{\sfc}{{\mathsf{c}}}
\nc{\sd}{{\Delta}}

 \def\e{\epsilon}

  \nc{\vol}{{\mathop{\operatorname{\rm vol\,}}}}
  \nc{\pr}{{\mathop{\operatorname{\rm pr\,}}}}
\nc{\co}{{\mathop{\operatorname{\rm Core\,}}}}
\nc{\comm}{{\mathop{\operatorname{\rm Comm\,}}}}
  \nc{\gal}{{\mathop{\operatorname{\rm Gal\,}}}}
  \nc{\Mod}{{\mathop{\operatorname{\rm Mod\,}}}}
  \nc{\modelsof}{\Mod} 
 \nc{\Aff}{{\mathop{\operatorname{\rm Aff\,}}}}
  \nc{\disc}{{\mathop{\operatorname{\rm disc}}}}
  \nc{\Sym}{{\mathop{\operatorname{\rm Sym}}}}
   \nc{\Alt}{{\mathop{\operatorname{\rm Alt}}}}
   \nc{\Aut}{{\mathop{\operatorname{\rm Aut}}}}
 \nc{\Spec}{{\mathop{\operatorname{\rm Spec}}}}
  \nc{\spec}{{\mathop{\operatorname{\rm Spec}}}}
\nc{\Ker}{{\mathop{\operatorname{\rm Ker}}}}
 \nc{\dom}{{\mathop{\operatorname{\rm dom}}}}
\nc{\End}{{\mathop{\operatorname{\rm End}}}}
 \nc{\Hom}{\operatorname{Hom}}
 \nc{\GL}{{\mathop{\operatorname{\rm GL}}}}
 \nc{\Id}{{\mathop{\operatorname{\rm Id}}}}
 \nc{\rk}{{\mathop{\operatorname{\rm rk}}}}
 \nc{\length}{{\mathop{\operatorname{\rm length}}}}
\nc{\supp}{{\mathop{\operatorname{\rm supp} \, }}}
\nc{\val}{{\rm val}}
 
\nc{\res}{{\mathop{\operatorname{\rm res}}}}

\def\Ind#1#2#3{{#1} {\downarrow}_{#3} {#2} }

\def\meet{\cap}
\def\union{\cup}

\def\si{\sigma} \def\Si{\Sigma}
\def\g{\gamma}
\def\G{\Gamma}

\def\m{\smallsetminus}

\nc{\seq}[1]{\stackrel{#1}{\sim}}

\def\inv{^{-1}}
\def\beq#1{\begin{equation} \label{ #1}}
\def\eeq{\end{equation}}

\def\prf{\begin{proof}}
\def\pv{\end{proof} }
 \def\eprf{\end{proof} }

\def\acl{\mathop{\rm acl}\nolimits}
 \def\dcl{\mathop{\rm dcl}\nolimits}

\def\a{\alpha}

 \renc{\b}{{\beta}}

\def\Ind#1#2{#1\setbox0=\hbox{$#1x$}\kern\wd0\hbox to 0pt{\hss$#1\mid$\hss}
\lower.9\ht0\hbox to 0pt{\hss$#1\smile$\hss}\kern\wd0}

 \def\Om{\Omega}

\def\ee{\mathcal{E}}

\setlist[itemize]{leftmargin=*}

\begin{document}

\begin{abstract}We look at equivalence relations on the set of models of a theory   such that the class of equivalent pairs is itself an elementary class,
in a language appropriate for pairs of models. We call these $\mer$s, for short.
We provide many examples of  $\bimer$s, along with the first steps of a classification theory for them.   We characterize the special classes
of   $\bimer$s associated
with preservation of formulas,  either in classical first order logic or in continuous logic,  and uncover an intrinsic role  for the latter.   We bring out a nontrivial  relationship with interpretations (imaginary sorts), leading to 
a wider hierarchy of classes related to the preservation of reducts.   
We give results about the relationship between these classes, both
for general theories and for theories satisfying additional model-theoretic properties, such as stability.   

\end{abstract}

\maketitle

\section{Introduction} \label{sec:intro}
We will study equivalence relations defined on the set of models of a given theory defined via a set of logical formulas in a vocabulary appropriate
for pairs of models, $\mer$s for short.

A sentence in the language of pairs of models will usually not define an equivalence relation. There are some simple classes of sentences that do give   $\mer$s, and we pay particular attention to these. One family of $\mer$s is  associated with preservation of open formulas (a.k.a. definable relations) -- either formulas of first-order logic or  (``ydlept $\mer$s'') or open formulas 
of continuous logic (``yclept $\mer$s''). We also extend the ydlept $\mer$ to a hierarchy of families (``$n$-ydlept'', for $n$ a natural number). These can be described via the preservation of definable families of sets, or by considering ydlepts over imaginaries.
We also consider the  subclass of $\mer$s  defined by restricting the quantifier complexity of the formula in the two-model language defining
the $\mer$. 

One motivation for the study is a step towards model-theoretic analysis of families of mappings  between structures. We have  examples of $\mer$s that arise 
from considering natural function classes.
A second motivation  is to generalize  classification theory to the study of models ``up to an equivalence relation''.

Naturally, we are far from achieving any of these  goals. 
 We will show that continuous logic plays a natural role in the study of $\mer$s, in that every $\mer$ can be associated
with a continuous logic lower approximation.
We are able to  identify some connections between families mentioned above: for example,
$\mer$s that  are definable with restricted quantifier alternation must be given by an open formula. 
We point the reader to Theorem \ref{thm:aerelative}, which gives a classification of the $\mer$s that are given by reducts
in continuous logic.
The reader may wish to look ahead to Figure \ref{fig:classes}, which shows the classes we study, illustrating some connections.

We also establish some connections between traditional
model-theoretic criteria on the underlying theory and properties of the corresponding $\mer$s. In stable theories one family of $\mer$s collapses to an a priori
smaller class. We show that theories defining a nontrivial partial ordering on their universe always admit non-yclept $\mer$s: see Proposition \ref{prop:yclept-unsop}.
The hypothesis is natural in view of Shelah's basic stability classification, where one of the classes is theories with a partial ordering on tuples (SOP).

Finally, we provide a large class of examples of $\mer$s, in the process showing non-connections  between  families.

{\bf Organization.}
The first part of the paper introduces our basic objects of study.
We give preliminaries on the logics we use in Section \ref{sec:prelims}, then introduce our main definition in
Section \ref{sec:defs}, $\mer$s that are definable either with a finite or an infinite set of sentences. The former will be called
$\bimer$s.
In Section \ref{sec:invar} we define a special class of $\mer$s given by preserving continuous logic definable sets -- the yclept $\mer$s.
And similarly for classical logic, the ydlept $\mer$s. We also introduce a hierarchy of $\mer$s that extend ydlept $\mer$s, the
$n$-ydlepts for $n \geq 1$.
In Section \ref{sec:maximalcl} we show that each $\mer$ definable by infinitely many sentences
is associated with a maximal yclept $\mer$. This \emph{maximal CL reduct} often gives significant information about the $\mer$.

The second part of the paper deals with relationship between the families of $\mer$s defined previously.
In Section \ref{sec:cldl} and  Section \ref{sec:notydlept} we investigate the intersection of yclept $\mer$s and $\bimer$s.
The intersection of $n$-ydlept $\mer$s with yclepts, and more generally
$n$-ydlept $\mer$ coarsenings of an yclept $\mer$, is
 examined in Section \ref{sec:yclepinthigherydlept}.
 In Section \ref{sec:stable} we show that when we introduce a classical model-theoretic restriction,
 stability, the hierarchy of $n$-ydlept $\mer$s collapses.
Section \ref{sec:aecomp} shows that if we assume that a $\bimer$ has a defining  formula
of low quantifier complexity, then we can make strong conclusions about where it fits in the hierarchy of families.
Section \ref{sec:smooth} utilizes the results of the prior section to give several characterizations of yclept $\mer$s.

The final part of the paper moves from relationships between restricted classes to explore generalization of traditional classification theory to the setting of model equivalence relations.   
While most of the results in the paper concern the relationship between $\mer$s within restricted families, like the yclept $\mer$s,
Section \ref{sec:basis} give some results about what structure in a theory is implied by the presence of a non-trivial $\mer$, while Section \ref{sec:theories} provides some preliminary results concerning theories  all of whose $\mer$s 
are yclept, making a connection to the strict order property.

In Section \ref{sec:examples}, we present a larger number of examples.
We close with further questions in Section \ref{sec:questions}.

\section{Preliminaries} \label{sec:prelims}

We often refer to two-valued first order logic as DL (discrete, or disconnected logic), in the context of formulas
or  theories, contrasting
it with Continuous Logic (CL).  In this work, unless otherwise specified,   $T$ will denote a DL  theory in language $L$. 
The word \emph{``definable'' will always mean, by default, definable with a formula without  parameters}.

Most of our results require only basic logic.   Occasionally we will also discuss statements  specific to certain tame classes in Shelah's classification, and for these we will  assume some basic familiarity:
\begin{itemize}
\item Stable theories, \cite{poizatmodeltheory}. 
\item NIP theories and the strict order property \cite{nipbook}
\item o-minimal theories (which are NIP) \cite{ominimalbook}
\end{itemize}
In fact we use little more than their definitions; NIP occurs only in examples.   Stability enters in \secref{sec:stable} via  the more general notion of  {\em stable embeddedness}, see e.g.  \cite[3.1]{nipbook}.

We will make use of continuous logic (CL) as in 
\cite{itayalexwardcontinuous}, but ``unmetrized''
(as in \cite{ckcontinuous}).  Equivalently
we can use the foundations of \cite{itayalexwardcontinuous} but assume a trivial metric. This drastically reduces the complexity of the metatheory. Continuous logic works over models in which predicates take
truth values in a bounded subset of the reals.

Formulas are built up from real-valued predicates 
using composition with continuous functions from products of reals to the reals, along with the quantifiers
sup and inf. Each such formula defines a bounded real-valued function on a model. We also close under
uniform limits: if $F_i(x)$ are real-valued predicates that converge for all $x$, their limit is a new real-valued predicate.

In much of the paper, we will be interested in CL over classical models -- the truth values of the predicates are just $\{0, 1\}$, but the formulas can take more general real values. The reader might look ahead to Example \ref{ex:CLexample} and 
Example \ref{ex:easyycleptnotydlept} for typical uses of CL.
While in most of the paper we will need little more than the semantics of CL formulas,
in  Section \ref{sec:smooth} and Section  \ref{sec:notydlept} we will make use of the notion of a \emph{CL theory}: this is a collection of assertions
of the form $\phi=0$, where $\phi$ is a CL sentence. The notion of a model satisfying such a theory is the obvious one.   We crucially use type spaces as well;  they are compact Hausdorff but not necessarily totally disconnected, in fact they may be connected.  


Compactness forms an essential feature of the theory of $\mer$s. We will use it  both 
directly and via the usual device of 'sufficiently saturated models'. In our setting, since we work with expansions throughout, the notion of {\em resplendency} provides a well-suited precise measure of the degree of saturation required.  We use it  for countable models.  We define a countable $L$-structure $M$ to be {\em resplendent} if  for any finite $A 
\subset M$, for any  expansion $L'$ of $L_A$ 
by finitely many symbols,   any recursive $L'$-theory $T'$ such that $T_A \union T'$ is consistent is realized in an expansion of $M$.


Note that for a countable model $M$, 
resplendence is equivalent to recursive saturation: any recursive partial type over a finite subset of $M$ realized in some elementary extension is already realized in $M$. 
See \cite{resplendent}. We leave the transposition to continuous logic to the reader.
For any recursively saturated $M$, there exists a Borel map $b$ from pairs $(M,T)$, with $M$ a recursively
saturated model and $T$ a finite 
theory consistent with $Th(M)$ to expansions of $M$, such that $b(M,T) \models T$. 
See the proof of Proposition \ref{prop:reducts} for one way to see this.

\section{Definition of model equivalences and basic properties} \label{sec:defs}
Let $L$ be a vocabulary, possibly with multiple sorts.  We let $S_\coupled$ be a subset of the sorts of $L$: the \emph{coupled sorts}. The other
sorts are the \emph{decoupled sorts} of $L$.

We define a new language  $\lequalslprimea{S_0}$.
It will have two copies of the decoupled sorts,  unprimed
and primed,  and only one copy of the coupled sorts. It will also have
two copies of each non-logical symbol of $L$.
Suppose we have a relation $R(x_1, \ldots x_n)$ of $L$, where $x_i$ has sort $S_i$. Then
in $\lequalslprime$ we have a relation $R(x_1 \ldots x_n)$, where each $x_i$ has the unprimed  sort corresponding to $S_i$,
along with a relation $R'(x_1 \ldots x_n)$ where $x_i$ has sort $S_i$ if $S_i$ has a coupled sort, while $x_i$ has sort $S'_i$ if $S_i$
is a decoupled sort.
The intuition for $\lequalslprimea{S_0}$ is that it describes a pair of $L$-structures, where on the coupled sorts they are forced to have the
same universe, while on the decoupled sorts they are disjoint.

We now formalize this intuition.
Let  $T$ be a first order $L$ theory
Let $\Om$ be a universe for each coupled sort of $L$. By $\modelsof_\Om(T)$ we denote the models of theory $T$ where the universes  of the coupled
sorts match $\Om$. A pair of models in $\modelsof_\Om(T)$  can be made into a model for 
$\lequalslprimea{S_0}$.
  in the obvious way, just by priming
the relations and the decoupled sorts in the second model.
Thus we can talk about a pair of models satisfying a sentence in   $\lequalslprimea{S_0}$.

We consider  
the situation where a theory $\tau$ in $\lequalslprimea{S_0}$ 
defines an equivalence relation on $\modelsof_\Omega(T)$, for any set $\Omega$. 
We say that $\Xi$ is a $\mer$.   If $\tau$ is 
finitely axiomatizable,  we talk of a 
\emph{bidefinable Model Equivalence Relation} or just $\bimer$.
Thus in case $T$ has no finite models, it suffices to check the definition of a $\bimer$ on a single set $\Omega$ of size at least $|L|+\aleph_0$;  and similarly for $\inftybimer$s.   Only $\inftybimer$'s (and among them, $\bimer$'s) will be considered from here on.

For most of the result in this work, it suffices to deal with the case where all sorts of $L$ are coupled: we refer to
these as \emph{completely-coupled} $\mer$s, and this will be our default. 
We refer to the language for pairs in this case as $\lequalslprime$.
While when we want to highlight the more general case with some decoupled sorts, we refer to  \emph{decoupled $\mer$s}.
In the completely-coupled case we will sometimes assume a single sort for simplicity.

We note that when the language is finite, we can basically assume the theory is finite as well:

\begin{lem} \label{finiteT}   Let $T$ be a theory in a finite language $L$. 
Let $\ee$ be a  $\bimer$ for $T$.  Then
there exists a finite $T_0 \subset T$ and a  $\bimer$  $\ee_0$ on models of $T_0$, such that $\ee_0 \upharpoonright \modelsof(T) = \ee$.
\end{lem}

\prf  We can assume that $L$ consists of a single relation $R$.  \\
So $\lequalslprime$ is generated by $R,R'$ and the equality across models.  We write $T(R')$ for the theory $T$, with $R$ replaced everywhere by $R'$.
We abuse notation to identify $\ee$ with its defining formula.
We have:  $  \ee(R,R')  $ is symmetric, reflexive and transitive on $\lequalslprime$-models of $T(R) + T(R')$.    By compactness, for some finite 
$T_0 \subset T$, $\ee(R,R')$ is symmetric, reflexive and transitive on $\lequalslprime$-models of $T_0(R) + T_0(R')$.
\eprf 

Actually the finite language assumption is unnecessary in \lemref{finiteT}; we leave this to the reader.

\ssec{Correspondence with definable groupoids}  

A {\em groupoid} is a category where every morphism is invertible.  
\label{defgr}  All groupoids considered here will have the same class of objects, namely $\modelsof(T)$; 
so they will be distinguished by their morphisms.  
We have the groupoid $Iso_T$ of models of $T$ with isomorphisms.  We also have the larger groupoid $Bij_T$ whose objects are models of $T$, and whose morphisms $Bij_T(M,N)$ are all bijections $M \to N$.
By a $T$-groupoid we mean an intermediate groupoid $\groupoid$, with objects $\modelsof(T)$ and
morphisms $\groupoid(M,N)$ with $Iso_T(M,N) \subset \groupoid(M,N) \subset Bij_T(M,N)$.

We consider a language describing two $L$-structures, along with bijections between 
their coupled sorts $S_0$.   A groupoid $\groupoid$ is {\em $\bigwedge$-definable} if there exists
a set $\Theta$ of sentences of $\lpluslprimea{S_0}$ whose models are precisely the 
$\groupoid$ morphisms.   If $\Theta$ is finite, we say $\groupoid$ is definable.

In more detail, let $S_0$ again represent the coupled sorts.   

By $\lpluslprimea{S_0}$ we denote the language with two copies of each  sort of $L$,
referred to as primed and unprimed.  For each relation or function in
$L$, $\lpluslprimea{S_0}$ has two copies, 
one taking the unprimed sorts and the other taking the primed sorts.
We have additional function symbols, $f$, going from the unprimed copy of a sort in $S_0$ to its primed copy.
A {\em bidefinable groupoid} $\CG$ over coupled sorts $S_0$
 is given
by a formula $\Theta$ in $\lpluslprimea{S_0}$,  where
 $\groupoid(M,N) = \{f: (M,N,f) \models \Theta \}$.   If  $\Theta$
 is allowed to be a possible infinite set of formulas,  we refer
to an {$\bigwedge$-definable groupoid}.   We write $Th(\groupoid)$ to denote the $\lpluslprimea{S_0}$-partial type defining 
groupoid $\groupoid$.
Thus $Th(\groupoid)$ is the $\lpluslprimea{S_0}$ theory of $\{(M,N,f):  M,N \models T, f \in \groupoid(M,N) \}$.
When all sorts are coupled we refer simply to $\lpluslprime$.

\begin{warning} \label{bipartitewarning}
  The definition of bidefinability and $\bigwedge$-bidefinability used in groupoids
applies even if it happens that $M=N$.   It is   not  the same as definability in the structure $(M,f)$, since
in our definition the sorts of the two components are distinct, so we cannot compare
(e.g.) $x$ and $f(x)$ directly.    Put another way, equality in the $\lequalslprime$ setting is being replaced by
identification via $f$ in $\lpluslprime$, so we do not have an additional equality.
\end{warning}

In the rest of this section, we will assume for simplicity a single-sorted completely-coupled setting. That is,
$L$ has a single sort, this is the coupled sort. 
In this case there is a single function symbol in $\lpluslprime$.

\begin{prop} \label{prop:equivrelationgroupoid}There is a natural 1-1 correspondence between  $\bimer$s for $T$,
and   bidefinable groupoids for $T$, 
and similarly for $\inftybimer$s and {$\bigwedge$-definable groupoid}s.
 \end{prop}

\prf  
Let $\Psi$ be a $\bimer$ for $T$.
Define a $T$-groupoid $Gr(\Psi)$ by letting   the morphisms $Gr(\Psi)(M,M')$ be the set of bijections $f:  M \to M'$ such that $M+ f ^*M' \models \Psi$,
where $M+ f ^*M'$ is the $L=L$ structure with the same universe as $M$ such that the interpretation of the first copy of a given $R$ is $R^{M}$,
and of the second copy, $f \inv (R^{M})$.  

Conversely, given a $T$-groupoid $\groupoid$, let $\Psi_\groupoid$ be the set of double models $(M',M'')$ of $T$ with the same universe $M$, such that $Id_M \in \groupoid(M',M'')$.  
It is easy to see that these operations are inverse to each other, and that $\Psi$ is bidefinable iff $\groupoid$ is bidefinable.

\eprf

The following proposition indicates that the behavior of the groupoid is in some sense determined by its restriction to a single model.

\begin{prop}\label{prop:groupoidsubgroup}  Assume $T$ is complete.     Let $M$ be any resplendent model of $T$, countable or not. 
Then there is a canonical $1$ to $1$ correspondence between $\bigp$s for $T$  and  
bidefinable subgroups of $Sym(M)$ containing $Aut(M)$.  Likewise for {$\bigwedge$-definable groupoid} and $\bigwedge$-bidefinable subgroups. \end{prop}

See Warning \ref{bipartitewarning} regarding the meaning of ``bidefinability'' even when the two parts are the same model.

\prf In one direction the correspondence is obvious:  a  {$\bigwedge$-definable groupoid} $\groupoid$ maps to $\groupoid(M,M) \leq Sym(M)$.   Let us show that this is one to one.
  Let $\groupoid$  be a   {$\bigwedge$-definable groupoid},
and  assume $\groupoid(M,M)=H$ is known.  We show that $\groupoid$ is the unique {$\bigwedge$-definable groupoid} groupoid inducing $H$ on $M$.    Given any  $N_1,N_2 \models T$ and bijection $b: N_1 \to N_2$, one can find  $(M_1, M_2,b')$   elementarily equivalent to $(N_1,N_2,b)$, as $\lpluslprime$ structures,
and with $M_1$ and $M_2$   saturated of the same cardinality as $M$.  Then $M_i \cong M$ and we 
have $(M_1, M_2,b') \cong (M,M,f)$ for some bijection $f: M \to M$.  
We have $b \in \groupoid(N_1,N_2)$ iff $b' \in \groupoid(M_1,M_2)$
iff $f \in \groupoid(M,M) = H$.  Thus the values of the groupoid on an arbitrary pair of models $N_1,N_2$ are determined
by $H$.

It remains to prove surjectivity.  Let $H$ be a  $\bigwedge$-definable subgroup of $Sym(M)$ containing $Aut(M)$.  
Then $H$ is defined by some collection $\Psi$ in $\lpluslprime$.  Define $\groupoid$ by
\[\groupoid(N_1,N_2) = \{f:N_1 \to N_2:  (N_1,N_2,f) \models \Psi  \}   \]
So $\groupoid(N_1,N_2)$ is a $\bigwedge$-definable set of functions $N_1 \to N_2$, but we have to check it is a groupoid:     if $f \in \groupoid(N_1,N_2)$
and $g \in \groupoid(N_2,N_3)$ then $g \circ f \in \groupoid(N_1,N_3)$.  And similar statements for inverses and for the identity maps.
Suppose otherwise, and take a saturated $(M_1,M_2,M_3,f',g')$ elementarily equivalent to $(N_1,N_2,N_3,f,g)$
and of the same cardinality as $M$.  Then $M_1 \cong M_2$ by completeness, so we may assume $M_1=M_2=M$.  In this case we have $f',g' \in H$
but $g' \circ f' \notin H$, a contradiction.   

It is clear that going in this way from a subgroup $H$  to a groupoid  $\groupoid$ and back  returns the same subgroup  $H$.

 The other direction follows from the fact, shown in the first paragraph, that $\groupoid \mapsto H$ is 1-1.  For the same reason, 
the choice of defining formulas $\Psi$ does not matter, so the correspondence is canonical.

\eprf

\begin{rem}  It is remarkable that in the setting of $\mer$s, a definable equivalence relation automatically gives a definable groupoid.

Similarly, a definable partial ordering on models  automatically gives a definable category.  

Assume for simplicity that $T$ is complete and that 
$\kappa$ is a cardinal with $\kappa^{+} =2^\kappa$; let $M$ be a saturated model of size $\kappa^+$.
We  saw that $\inftybimer$ is determined by a subgroup $H$ of   $Sym(M)$  containing
 $Aut(M)$.   $H$ is not in general closed in the usual topologies on $Sym(M)$; however it does enjoy a closure property similar to the one Lascar introduced in \cite{lascar}, in defining the topology on the Lascar group.   Namely if $U$ is an  ultrafilter on $\kappa$, $(M^*,H^*)$ is the $U$-ultrapower of $(M,H)$, and   $f: M \to M^*$ 
is an isomorphism,  we have $f \inv H^* f \leq H$.

 Investigating this further would be very interesting. 
\end{rem}

\section{Special classes of  equivalence relations} \label{sec:invar}

 We will
look at equivalence relations defined over  models of $T$ using preservation of a set of CL or DL formulas.

A \emph{DL reduct} of a theory $T$ is given by a collection of formulas of discrete logic and the associated theory.
Given a set of relation symbols, we consider the DL reduct given by these relations as atomic formulae. Thus
a language $L$ can be considered a vacuous reduct of itself.
A CL reduct is defined analogously.

Now suppose we are given a reduct $R$ on the coupled sorts within $L$, given by formulas of the corresponding logic. We let $\equiv_R$ be
the equivalence relation saying that two structures agree on these formulas, and $\Aut(R)$ the
corresponding groupoid, which consists of mappings $g$ taking a model $M$ to a model $M'$.
Thus, for a CL reduct given by  a collection of formulas $F$, we are saying that each the function defined by $f \in F$ on model $M$ is the same as the function
defined on $M'$.

When the theory is incomplete, the reduct may include $0$-place relations, corresponding to sentences the theory does not decide.   To say that these relations agree on two $L$-structures just means that they have the same truth value.    

\begin{defn}
We will say that a model equivalence relation is  \emph{yclept}  
if it has the form $\equiv_R$ for some  continuous logic reduct $R$. And we say that it is \emph{ydlept} 
if $R$ consists of standard first-order two-valued
formulas: that is,  it is a ``disconnected'' or ``discrete''  or '`classical'' first-order logic reduct.  \end{defn}

Note that \emph{a reduct only defines a $\mer$ when its vocabulary  is on the coupled sorts, since on the decoupled sorts we cannot
talk about two formulas having the same tuples}. Often in the sequel, we will specify an yclept or ydlept by simply listing
the formulas or relations that are preserved. The default convention is that the coupled sorts are all the sorts.
Note also that a ydlept $\mer$ is not necessarily definable, since it can involve preservation of infinitely many first order formulas.

A yclept $\mer$ can also involve either finitely many or infinitely many CL formula. But we observe that
any yclept equivalence relation, even one involving infinitely many formulas,   is $\bigwedge$-definable:
\begin{lem} The equivalence relation corresponding to preserving a CL reduct is a $\inftybimer$.
\end{lem}
\begin{proof} The reduct includes certain uniform limits $f(x)$ of $T$-definable functions $f_n(x)$ with finite range in $[0,1]$.  We
may assume $|f-f_n| \leq 1/n$ uniformly.  
Thus  the groupoid consists of triples $(M, M',g)$ such that $M \models T, M' \models T$
and $(\forall x)(|f_n(x) - f_n(g(x))| \leq 2/n)$ for each such $f$ and $n$.   Thus
it is defined by infinitely many DL sentences. \end{proof}

\begin{lem} \label{prop:finitelymany} Let $\ee$ be an ydlept $\mer$.  Then $\ee$ is a $\bimer$ if and only if 
it has the form $\equiv_R$ for $R$ a finite collection of   first-order formulas.
\end{lem}
\begin{proof}   One direction is immediate.  For the other, we have by ydleptness that $\ee$ is equivalent to a conjunction of
$\equiv_{R_i}$, $i \in I$.  
Suppose, for each finite $I_0 \subset I$, we can find inequivalent $M,N$ on the same universe, such that 
$R_i^M =R_i^N$ for $i \in I_0$.  Taking an ultraproduct (using definability of $E$), we find $M,N$ that are inequivalent,
yet $R_i^M =R_i^N$ for $i \in I$.  This contradiction shows that we may take $I$ to be finite, say $I=\{1,\ldots,n\}$.
\end{proof}

An extension is to look at defining equivalence relations via \emph{definable families}.  

Let $\phi(\vec x_1 \ldots \vec x_{n+1})$ be a formula, possibly with parameters from the model, with $\vec x_i$ a partition of the free variables.
We define the $(n+1)$-set associated with the partitioned formula by induction on 
$n$, denoted $[\phi(\vec x_1 \ldots \vec x_{n+1})]$.
If $n=0$ it is just the single set of satisfiers of $\phi$. For $n>0$, the $(n+1)$-set  is the family
of sets $[\phi(\vec x_1 \ldots \vec a_{n+1})]$ as $a_{n+1}$ varies over the model.

For a definable set $R$, let $\ee_n([R])$ be the  equivalence relations of preserving the $n$-set $[R]$.   
If $R$ is defined by $\phi(\vec x_1 \ldots \vec x_{n+1})$, the $n$-set is an element of the iterated power set
over the sorts of $\vec x_1$. Thus it is natural to take the coupled sorts of this $\mer$ to be the sorts of $\vec x_1$.
When  we want to emphasize the particular
sorts $S_0$  that the higher-order object is built on 
we talk about an $n$-set \emph{over} sort $S_0$.
An $n^{th}$ order ydlept equivalence relation is given by preserving some collection of $n$-sets.
An $n^{th}$ order ydlept definable equivalence relation is thus given by preserving a finite set of $n$-sets: it is easy to see that these
can be consolidated into a single $n$-set.

We give another definition of higher-order ydlept that will be more amenable to inductive proofs.

We   recall the usual $^{eq}$ procedure, in a form that will be more convenient here.   Let $D$ be a definable subset of some finite product 
$S_1 \times \cdots \times S_n$  of sorts of a language $L$.
Let  $F$ be a definable family of  subsets of $D$.  Thus for some definable $Q$ and definable $R \subset D \times Q$, we have 
\[ F = \{R(d'): d' \in Q \} \]

\begin{defn} \label{eq} Let $T$ be a theory in a language $L$, and $F$ a definable family as above.   We define the 
   ``Shelahization'' $T_F$ of $T$ at $F$.   
We let $L_F$ be the language $L$ with an additional sort $S_F$, and  a new relation $C_F$ that relates elements of $S_F$ to elements
of $D$. Abusing notation, we will sometimes  treat $C_F$ as  a function from  elements of the sort  to subsets of $D$.
 The theory $T_F$ includes $T$ along with  the statement that 
\[ s \mapsto C_F(s) : \ \ S_F \to F \]
is a bijection.  \end{defn}

Note that a model of $T$ extends canonically to a model of $T_F$, and every model of $T_F$ is obtained in this way.  Definable sets of $T_F$
can be understood in terms of definable sets in $T$.

 \begin{defn}  \label{def:nydlept}  
We define the class of $n$-ydlept equivalence relations by induction on $n$. The base case, $n=1$,  
are the ydlept $\bimer$s.
  An $(n+1)$-ydlept  equivalence relation on models with a given universe is determined by a definable family
${\mathcal F}$, say  given by $\phi(\vec x_0, \vec p)$ with $\vec x_0$ having
sorts $S_0$, along with a \emph{well-behaved}  $n$-ydlept equivalence relation $\ee$ on $T_{\mathcal F}$ with respect to 
 coupled sorts $S_1$ that contain: the sorts $S_0$ and   the additional imaginary sort $S_F$, where   well-behaved means that
they preserve the new relation $C_F$. The equivalence relation will have coupled sorts $S_1$, and
we declare models $M,M'$ of $L$ that agree on the universe for $S_1$ to be  equivalent iff
the $2$-sets given by $\mathcal{F}$ are the same in both models,  and the expansions of both models  are $\ee$ equivalent.
\end{defn}

Thus an $n$-ydlept is formed by first taking a definable family to create an equivalence relation, and identify equivalent tuples. Then take
a definable family in the quotient, and so forth.
As with $n$-sets, we can talk about an $n$-ydlept on sort $X$, where we iterate the process above where in every iteration the
$L$ sorts used are in $S$.

It is easy to see that every $n$-set induces an $n$-ydlept. The fact that a $2$-set is  a $2$-ydlept is almost by definition.
For a $3$-set given by partitioned formula $R(\vec x_1; \vec x_2; \vec x_3)$, we first use the definable family $R(\vec x_1, \vec x_2; \vec x_3)$
to create an  imaginary sort with elements
$o$ corresponding to tuples $\vec x_1, \vec x_2$. We then define a $2$-ydlept on the corresponding
expansion, via a family indexed by $\vec x_1$ and defining all the  $o$'s corresponding to $\vec x_1, \vec x_2$.
Similarly for higher values of $n$.

We will now  show the converse: every $n$-ydlept  $\bimer$ is induced by an $n$-set.
We start by showing that every $2$-ydlept is a $2$-set.

\begin{lem} \label{lem:2ydleptand3set}
Let $L$ be a language, $S$ a subset of the sorts
of $L$, $\phi_1(\vec x_1, \vec p)$ be  a partitioned $L$ formula, representing a definable family, with $\vec x_1$ having sorts in $S$.
Let $L_1$ be formed from adding
a sort $S_1$ and a relation $C_1$ mapping $\vec x$ to the corresponding  canonical parameters for $\phi_1$.
Consider an ydlept $\bimer$ $\ee$ over $L_1$ with coupled sorts $S \cup S_1$, given by $\phi_2(\vec x, \vec o)$, where $\vec x$ are also over sorts in $S$
and
$\vec o$ are over sorts in $S_1$. 
Then there is a $2$-partitioned $L$-formula $\phi_3(\vec x_1, \vec p_1)$ such that for $M,M'$ $L$-structures that agree on sorts $S$,
$M$ and $M'$ agree on the $2$-set defined by $\phi_3$ if and only if they agree on the $2$-set $\phi_1$, and,
letting $M_1$ and $M'_1$ be their canonical extensions with $S_1$ and $C_1$, the models $M_1$ and $M'_1$ agree on the definable set $\phi_2$.
\end{lem}

\prf    For simplicity we consider the case $\phi_2(x, o)$: one ordinary variable $x$ and one imaginary variable $o$.
Now the $2$-ydlept $\ee$ is the equivalence relation setting $M$ equivalent to $M'$ if and only if $M$ and $M'$ agree on 
the $2$-set given by  $\phi_1$ and also on $D_{2,1}= \{ x, \vec p ~ | ~ \phi_2(x, C_{1}(\vec p))\}$. 
But since the addition of imaginaries gives no new $0$-definable sets on the original $L$-structure,
$D_{2,1}$ is a definable set.
\eprf

The following  is proven similarly, using induction.

\begin{lem}   \label{lem:special} Every $n$-ydlept $\bimer$ $\ee$ is induced by an $n$-set, and similarly every an $n$-ydlept $\mer$ is induced by a collection of $n$-sets.
\end{lem}

We say that a $\mer$ is $\omega$-ydlept if it is $n$-ydlept for some $n$.
We have now defined the main subclasses that will be considered in the paper.
Figure \ref{fig:classes} shows them schematically.
Note that in the figure the intersection of yclept with $\omega$-ydlept is ydlept. This is Corollary
\ref{cor:ycleptcaphigherydleptimpliesydlept} that will be proven in Section \ref{sec:yclepinthigherydlept}.
On the other hand, the figure shows the intersection of yclept with $\bimer$ as being bigger than ydlept.
This will be shown in Section \ref{sec:notydlept}.

  \begin{figure}[h!] \label{fig:classes}
 \includegraphics[width=8cm]{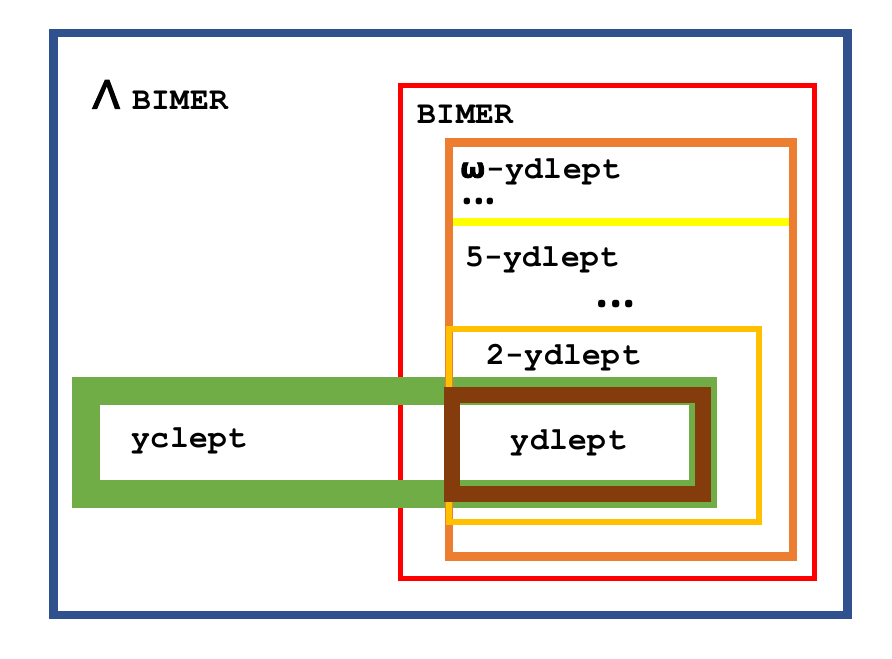}
 \caption{Classes of model equivalence relations}
 \end{figure}

\section{The maximal CL reduct} \label{sec:maximalcl}

We study here the maximal definable structure preserved by a {$\bigwedge$-definable groupoid}.  Our proofs, 
as well as the main statements, require 
 working in continuous logic,  even if $T$ is a DL theory; see  Example \ref{ex:CLexample}.
In case $T$ is small, any CL reduct is essentially DL so that the statements are true for
the maximal DL reduct; however our proofs still go through
continuous logic.

For a groupoid $\groupoid$ and theory $T$, we say that a model of $T$ is
\emph{ $\groupoid$-resplendent} if  $M$ is $\aleph_0$-homogeneous, and whenever $a,b \in M^n$ and the $\lpluslprime$ partial type
$tp_M(a)(x) + tp_M(b)(y)  +(M,M,g) \models Th(\groupoid) + g(x)=y$ is consistent,  there exists
$g \in \groupoid(M,M)$ with $g(a)=b$.   Note that if $L$ is countable, any countable $M_0$ has a $\groupoid$-resplendent countable
elementary extension.   In fact, if  $M$ is countable and recursively-in-$Th(\groupoid)$-saturated, then
$M$ is $\groupoid$-resplendent.  If $\groupoid$ is definable, ordinary resplendence suffices.

\begin{prop} \label{prop:maxred}   Let  $T$ be a first order theory in a language $L$, and $\groupoid$ be a {$\bigwedge$-definable groupoid} for $T$.    Then: \begin{enumerate}
\item  There exists a unique richest continuous logic reduct $\maxred{\groupoid}$   of $T$,  such that,
$Iso_T = {Iso}_L  \leq \groupoid  \leq {Iso}_{(\maxred{\groupoid})}$.   
 
\item  $\groupoid$ is dense in each automorphism group of the groupoid corresponding to $\maxred{\groupoid}$, taken with the pointwise convergence topology; i.e.:    for $M$ $\groupoid$-resplendent,
the orbits of $\groupoid$ and of  ${Iso}_{\maxred{\groupoid}}$ on $M^n$  coincide for each $n$.

\item  Suppose $T$ is small: only countably many  $n$-types over $\emptyset$
for each $n$.  Then ${Iso}_{\maxred{\groupoid}}$ is  ydlept: that is, is given by a DL reduct.
\end{enumerate}
\end{prop}

\prf  We start by proving (1). Let $L^c$ consist of all continuous logic relations, that are $\groupoid$-invariant.
Here continuous logic relations are viewed as $0$-definable functions   $F: X \to [0,1] \subset \Rr$ where $X$ is a finite product of sorts of $T$, and  $0$-definable means that $F \inv(U)$ is $\bigwedge$-definable for any closed $U \subset [0,1]$.
   $\groupoid$-invariant means that if $g \in \groupoid(M,N)$ then for any $x \in X$, $F(x)=F(gx)$.  
It is clear that $L^c$ is the richest continuous logic reduct $T'= T \upharpoonright {L^c}$  of $T$ with
$\groupoid  \leq {Iso}_{L'}$. 

We turn to (2),  density. Let $S_n$ be the space of types of $T$ on $X^n$.    
 Define an equivalence relation on $S_n$:  $p \sim q$ iff there exist $M \models T$ and $a \models p, b \models q$ in $M$ and 
$f \in \groupoid(M,M)$ with $f(a)=b$.  This is a closed subset of $(S_n)^2$  by an ultraproduct argument:   if $p_i \sim q_i$ for $i \in I$, this is witnessed by $f_i \in \groupoid(M_i, M_i)$.  
Let $U$ be any ultrafilter on $I$. Let $M$ be the ultraproduct of $M_i$
let $p$ be the limit of the $p_i$ along $U$, and similarly $q$. Let $f: M \to M$ be the ultraproduct of $f_i: M_i \to M_i$.  Then since $\groupoid(M,M)$ is 
$\bigwedge$-bidefinable,
$f \in \groupoid(M,M)$.  

Note that $p \sim q$ iff for  some/any $\groupoid$-resplendent $M \models T$ that realizes $p,q$, for some / any $a \models p$ and $b \models q$ there exists   $f \in \groupoid(M,M)$ with $f(a)=b$.  The second passage from ``some'' to ``any'' uses the fact that $Iso_T \leq \groupoid$, composing  with an automorphism  taking one realization of $q$ to another.  It follows that $\sim$ is  an equivalence relation.   In the case $n=0$,  $p,q$ are completions of $T$, and $p \sim q$ iff for some resplendent pair $(M,N)$ we have $p=Th(M)$, $q=Th(N)$,
iff this holds for all resplendent pairs of models  of $T$.

  Let $Y$ be the quotient space
$S/\sim$, $\pi: X \to Y$ the quotient map.  From the fact that
$\sim$ is closed in $X^2$, we conclude that $Y$ is a Hausdorff space,
and since it is compact it is thus a normal space, and by Urysohn's Lemma  points can be separated
by continuous functions.
Let $M \models T $ be resplendent, $a,b \in X(M)$,   and  suppose there is no $g \in \groupoid$ with $g(a)=b$.
Let $p =tp(a), q=tp(b)$.  Then $p \not \sim q$.  So $p,q$ have distinct images $p',q'$ in $Y$.  As mentioned just above, there
is a continuous function $\phi: Y \to [0,1]$ with $\phi(p') \neq \phi(q')$.  So $F=\phi \circ \pi$ is a definable function and it is $\groupoid$-invariant,
hence belongs to the reduct $L^c$; and $F(a) \neq F(b)$.   Thus $tp_{L^c}(a) \neq tp_{L^c}(b)$. 
We have shown that if two tuples agree on all definable relations in the maximal CL reduct, there is a groupoid element that maps from one to another. Thus we have proven density.

Finally, we argue for (3). For small $T$, by definition $S_n(T)$ is countable and hence so is the quotient. Countable compact spaces are totally disconnected, so
we can choose $\phi$ above valued in $\{0,1\}$: there must be a clopen subset $C$ of $Y$ with $p' \in C$ and $q' \notin C$, so
we let $\phi$ be the characteristic function of $C$.   This shows that we could take $L^c$ to consist of all $\{0,1\}$-valued $\groupoid$-invariant definable functions, i.e. an ordinary reduct.

\eprf

We can similarly refer to $\maxred{\modeleq}$ for the maximal CL reduct of a model equivalence $\modeleq$.
We also introduce notation for the theory of the maximal reduct.

\begin{defn} \label{def:languagetheoryofgroupoid}
For a $\bigwedge$-definable groupoid  $\groupoid$, we let $L^\groupoid$ 
be the set of CL-definable relations that are invariant under $\groupoid$. We let
$T^\groupoid$ be the restriction of $T$ to $L^\groupoid$.
\end{defn}

Consider a CL formula whose range is finite. Such a formula corresponds to a finite collection of pairs, each consisting
of a DL formula and a real number.
We call such a formula {\em essentially DL}, and similarly refer to an essentially DL reduct.

\begin{rem}\label{rem:CLDLy}  Assume $\ee$ is an yclept $\bimer$.  Then $\ee$ is ydlept iff the maximal CL reduct is essentially DL.
\end{rem}

\prf  Let $\ee$ be given by CL formulas $\Gamma$. The easy direction is to assume that the maximal CL reduct is essentially DL. Since the maximal CL reduct
will include $\Gamma$, we have that $\Gamma$ can be taken to be essentially DL, and thus the equivalence relation is generated by DL formulas.
Assume  $\ee$ is ydlept.   So it is the $\bimer$ corresponding to some DL reduct $L^d$.  We have $L^d \subset L^\ee$ as $L^\ee$ is maximal invariant.
And any model of $T'=T \upharpoonright L^c$ expands {\em uniquely} to $T \upharpoonright L^\ee$.   It follows by Beth's theorem that  all relations of $L^\ee$ are already definable over $L^d$.
\eprf

\begin{defn} \label{def:bimax} For a groupoid $\groupoid$ on models of a theory $T$,
we let $\bitheorymax$ denote  the theory of the equivalence relation generated by $\maxred{\groupoid}$, which we conflate below with the theory of the groupoid, the latter
consisting of triples $(M, M', g)$.
\end{defn}

\begin{lem} \label{lem:ae1} Let $\groupoid$ be an $\inftybimer$. Assume that $T$ has quantifier-elimination.
Then
every universal consequence $\alpha$ of the theory of $\groupoid$ in the language $\lpluslprime$ is a consequence of $\bitheorymax$.
 \end{lem}
Here, assuming $T$ has quantifier-elimination is equivalent to redefining ``universal'' $\lpluslprime$
formulas to be those that have only universal quantification across the primed and unprimed signature.
 \prf A consequence $\alpha$ can be taken to have the form $(\forall x)(\neg \beta(x,gx))$ where $\beta$ is
a formula of $L$ - indeed a Boolean combination of $L$ formulas in $x$ and in $g(x)$, which we  considered quantifier-free by convention, 
  and $x$ is a tuple of variables.
Consider any complete  type
$p(x)$ of $\maxred{\groupoid}$.  
 If  $p(x) \wedge \beta(x,y) \wedge p(y)$ is consistent with $T$,
it is realized in a resplendent model $M$ of $T$ by some $(a,b)$; but then by density of $\groupoid(M,M)$
in $Aut_{\maxred{\groupoid}}(M)$, \propref{prop:maxred},
there exists $\si \in \groupoid(M,M)$ with $\si(a)=b$. This contradicts the assumption that $\alpha$ follows from
the theory of $\groupoid$.  Hence $T \union \{\beta(x,y)\} \union tp_{\maxred{\groupoid}}(x)=tp_{\maxred{\groupoid}}(y)$ is inconsistent.

In case  $\maxred{\groupoid}$ is generated by DL formulas, by compactness, for some $L^{\groupoid}$-formulas $P_j$, we have  $T  \vdash (\forall x) \bigvee_{j=1}^m P_j(x)$ and $T \models P_j(x) \wedge P_j(y) \vdash \neg \beta(x,y)$.  Since
$\bitheorymax \models P_j(x) \iff P_j(gx)$,  it
 follows that   
$\bitheorymax$ proves $\alpha$ as claimed.  The argument in the case where $\maxred{\groupoid}$ is generated by CL formulas is similar:
we find $P_j$ and $\e$ such that 
$T$ is inconsistent with  $|P_j(x)-P_j(y) | \leq \e \wedge \beta(x,y)$.

\eprf

We give a  corollary of Proposition \ref{prop:maxred}.

\begin{cor} \label{cor:closedyclept} Let $\ee$ be a $\inftybimer$
 with groupoid   $\groupoid$.   Then $\ee$  is  yclept
 iff $\groupoid(M,M)$ is a closed subgroup of $Sym(M)$, for all  $\groupoid$-resplendent    $M$.

  In case $T$ is complete, $\ee$ is yclept
  iff $\groupoid(M,M)$ is closed for one resplendent $M$.
  \end{cor}

\prf   
For one direction, if $\ee$ is yclept,
then $\groupoid(M,M)$ is the subgroup of $Sym(M)$ fixing a certain family of finitary
relations, or functions $\rho$ into the reals. This is clearly closed: if $g$ does not fix $\rho$
then $\rho(a) \neq \rho(g(a))$ for some tuple $a$, and so the same inequality holds for
each $h$ with $h(a) = g(a)$, an open subset of $\groupoid$.

We prove the direction from right to left, assuming that  $\groupoid(M,M)$ is a closed subgroup of $Sym(M)$ for any resplendent $M$, and arguing that $\ee$ is yclept.
 Being closed and dense, by Proposition \ref{prop:maxred}, $\groupoid(M,M)$  must be the full  automorphism group of the maximal CL reduct
$\maxred{\groupoid}$.

Let $(M,N)$ be a   pair of models of $T$ on the same universe, and assume $M \upharpoonright \maxred{\groupoid} = N \upharpoonright
\maxred{\groupoid}$; we have to show that $M \ee N$.   By the remark following Proposition
\ref{prop:maxred} there exists $M' \ee M$ with $Th(M')=Th(N)$; thus we may assume 
$Th(M)=Th(N)$.  Further we may assume $(M,N)$ is resplendent; so $M \cong N$, say
via $g: M \to N$.   Since $M \upharpoonright \maxred{\groupoid} = N \upharpoonright 
\maxred{\groupoid}$, $g$ is a $\maxred{\groupoid}$-automorphism, and hence by the previous paragraph, $g \in \groupoid(M,M)$.  So $M \ee N$. 
 \eprf

\section{DL and CL reducts} \label{sec:cldl}

With the definitions behind us, we now begin the second part of the paper, where we address some basic questions about yclept $\mer$s versus $\bimer$s.
For example, we give properties of $\mer$s that are both  yclept and a $\bimer$, but show
that this does not  imply being ydlept.

In small theories  the distinction between CL and DL reducts is not critical. Let us show a converse
to this.  

  \begin{prop} \label{prop:smallchar}  Let $T$ be a theory  in a countable language.   Then
$T$ is   small iff   every CL reduct of an expansion of
  $T$ by finitely many constants can be presented as a DL reduct.

  Therefore  a small $T$ as well as expansions by  finitely many constants have
     the property that ydlept = yclept:
     if an equivalence relation  comes from a CL reduct
  then it comes from a DL reduct.    \end{prop}

 \prf   If $T$ is small, then every expansion by constants and every CL reduct of such an expansion is small. Hence every CL reduct is
 generated by the discretely valued formulas in it: the type spaces are totally disconnected.

 If $T$ is not small, some expansion by constants has uncountably many $1$-types. So we may assume this holds for $T$.   So
   $S_x$ is uncountable.  Then $S_x$ contains a perfect subset $P$.
 $P$ is homeomorphic to the Cantor set, and admits a continuous map onto the interval $[0,1]$.  By the Tietze extension theorem, there exists a continuous, surjective $F: S_x \to [0,1]$.  Now $F$ can be viewed as a  CL relation in the variables $x$.  Since $x$ is a single variable, the reduct generated by $F$ admits quantifier elimination, and actually
 $F(x)=\alpha$ generates a complete type, so that the space of $1$-types is homeomorphic to $[0,1]$.
 Hence every unary relation of the reduct to $F$ is a continuous function $C(F(x))$ of $F(x)$, and as $[0,1]$ is connected, if
 $C \circ F$ takes more than one value then it takes infinitely many.  So $F$ cannot be approximated by
 definable relations with finitely many values.   \eprf

As an alternative statement,  a theory is small iff in any bi-interpretable theory, every yclept $\mer$ is ydlept.   (Smallness is of course
preserved under bi-interpretation, giving the left-to-right direction.  In the opposite direction, assuming $T$ is not small,  it has uncountably many $n$-types on some product $S$ of sorts.  Taking a bi-interpretable theory with $S$ a sort, we obtain an yclept but non-ydlept reduct as above.

We now consider the intersection of $\bimer$s with yclepts.

 \begin{thm} \label{thm:approxequiv} 
Assume 
$\ee$ is  a $\bimer$ on models of theory $T$, and that $\ee$ is also yclept.

Then  there exists a CL
definable map $R$ on $M^k$ into a finite metric space $(Y,d)$, and some $\e>0$ such that
$M \ee N$ iff $d(R^M(a),R^N(a))< \e$ for all $a$. Note that  such a map is essentially DL, in the sense defined earlier.

Conversely, the existence of such a map  implies that $\ee$ is yclept.
\end{thm}
  One might say $T$ is approximately ydlept in the above case.

We note a corollary:
\begin{cor} Whenever we have  a $\bimer$ $\ee$ that is also yclept,  and we move to an expansion of $T$ with quantifier elimination, then
 $\ee$ is defined by a   universal  $\lequalslprime$ sentence.
 \end{cor}
 \prf
The theorem shows that $M \ee N$ iff for all $a$, a certain Boolean combination of the truth values of  $R^i(a),R^i(b)$ is valid, where $R^i$ are DL 
formulas.
\eprf

We now prove Theorem \ref{thm:approxequiv}.
\prf  We have some CL relations $(R_i)_{i \in I}$ such that $M \ee  N$ iff $R_i^M = R_i^{N}$ for each $i \in I$.
As in Proposition \ref{prop:finitelymany},  $I$ can be taken to be finite; say  $I=\{1,\ldots,n\}$.
We may replace $R_1,\ldots,R_n$ with a single $k$-ary relation taking values in $Y=[0,1]^n$.  Fix some metric $d$ on $[0,1]^n$,
say $d(x,y) = \sup |x_i-y_i|$.
By a 
compactness argument, for some $\e>0$,  if  $d(R(a)^M , R(a)^N)) < \epsilon$ for all $a \in M^k$,
then $M \ee N$, and hence $R(a)^M=R(a)^N$.

   There exists a DL definable function ${\delta_R}$ into a finite subset of $[0,1]^n$, such that
$d(R(x),{\delta_R}(x)) < \e / 8$ for all $x$.     We have $M \ee N$ iff  $d( {\delta_R}^M(x) ,{\delta_R}^N(x)) < \e/4$ for all $x$.

For the converse, we may assume $T$ admits quantifier-elimination.  Thus we can assume  $\ee$ is defined by a universal formula. The  quantifier-free
part is a Boolean combination of  formulas over both models along with equalities.
It is clear that a universal definable $\mer$ gives a groupoid that is closed. Hence by Corollary \ref{cor:closedyclept}
 $\ee$ is yclept.

\eprf

For much of this work we focus on complete theories. But the following results will allow
us to transfer characterizations of  ydlept and yclept equivalence relations from complete
theories to incomplete theories. We start with a result about ydlepts.

\begin{prop}\label{ydlept-conts}   Let $\ee$ be a $\bimer$ for a theory $T$.  If $\ee$ is ydlept when restricted to $\modelsof(T')$ for any completion $T'$
of $T$, then $\ee$ is ydlept.
\end{prop}

\prf Let $\maxreddl{\ee}$ be the maximal DL reduct associated with $\ee$, and let $\ee_r$ be the (ydlept) model equivalence relation of having the same   $\maxreddl{\ee}$ reduct.
We begin with an observation:

\begin{claim} \label{clm:commute}
Let $\Si$ be any set of sentences of $L$, closed under conjunctions.    The map 
$\ee \mapsto \ee_r$ commutes with
restriction to $\modelsof(T \union \Si)$. 
\end{claim}
\prf On the face of it $\maxreddl{\ee}$ may grow upon this restriction, since some new $\phi$
may be $\ee$-invariant on $\modelsof(T \union \Si)$.   But if this is the case, by compactness there exists a single $\si \in \Si$ responsible for it.
Then $\si \to \phi$ is $\ee$-invariant on $\modelsof(T)$. So $\si \to \phi $ is in $\maxreddl{\ee}$ and on  $\modelsof(T \union \Si)$ it is equivalent to
$\phi$. Thus $\phi$ is not really new.
\eprf

We now consider a special case of the proposition:

\begin{claim} \label{clm:restricted}
 The proposition holds in the case that $T \upharpoonright \maxred{\ee}$
is complete. 
\end{claim}
\prf
Accordingly, we assume $T \upharpoonright \maxreddl{\ee}$  is complete, and that $\ee$ is ydlept when restricted to $\modelsof(T')$ for any completion $T'$.
We will use for the first time the assumption that $\ee$ is definable, rather than just $\bigwedge$-definable.  So $\neg \ee$ is definable, and  \[ M \equiv M', \ \   M \upharpoonright \maxreddl{\ee} = M'|\maxreddl{\ee}, \ \  \neg \ee \] is inconsistent.
By compactness there are sentences $\si_1,\ldots , \si_k$ such that if $M \upharpoonright \maxreddl{\ee} = M' \upharpoonright \maxreddl{\ee}$, and $M \models \si_i $ iff $M' \models \si_i$,
then $M \ee M'$.

Let $X'$  be the space of (at most $2^k$) consistent
extensions of $T$ by $\si_i$ or $\neg \si_i$ for each $i$, which are complete for these $k$  sentences.   Since
$T \upharpoonright \maxreddl{\ee}$  is complete, we can fix some resplendent model $M_0$ of $T \upharpoonright \maxreddl{\ee}$, and consider expansions $M$ of $M_0$ to $T$.
Two such expansions are $\ee$-equivalent if they agree on each $\si_i$. So the $\ee$-class of such an $M$ is completely determined by
$Th_{\si_1,\ldots,\si_k}(M)$.  We can define an equivalence relation on $X'$:  $T' \sim T''$ if for some (equivalently all) $M' \models T'$
and $M'' \models T''$, both expansions of $M_0$, we have $M' \ee M''$.    Let $C_1,\ldots C_m$ be the classes of $\sim$ on $X'$.
For each class $C_j$ let $\tau_j = \bigvee_{T' \in C_j} \bigwedge_{\pm \si_i \in T'} \pm \si_i $.  Then $\sim$ preserves each $\tau_j$,
so each $\tau_j \in \maxreddl{\ee}$, and hence since any two expansions  $M' ,M''$ of $M_0$ agree on the $\tau_j$, they satisfy the same
$\tau_j$, and so by definition of $\sim$, they are $\ee$-equivalent.

Thus the ydlept equivalence relation we need is given by the sentences $\tau_j$ and by $\maxreddl{\ee}$.
\eprf

We now return to the proof of the proposition.
Assume $\ee$ is ydlept when restricted to any completion of $T$.   Let $X$ be the space of
extensions of $T$ by $\maxred{\ee}$-sentences which are complete for $\maxred{\ee}$ sentences. 
  By  the claim,  $\ee$ is ydlept when restricted
to any  $T' \in X$. Thus by Claim \ref{clm:commute}  $\ee = \ee_r$ on $\modelsof_\Omega(T')$.  But if $M,M' \in \modelsof_\Omega(T)$ and are $\ee_r$ equivalent,
automatically they both lie in $\modelsof_\Omega(T')$ for some $T' \in X$, namely $T'=Th(M \upharpoonright \maxreddl{\ee}) = Th(M' \upharpoonright \maxreddl{\ee})$.    So $\ee=\ee_r$ on $\modelsof_\Omega(T)$, and
we have proven the proposition.

\eprf

\begin{prop}\label{yclept-conts}   Let $\ee$ be a $\bimer$ for a theory $T$.  If $\ee$ is yclept when restricted to $Mod(T')$ for any completion $T'$
of $T$, then $\ee$ is yclept.
\end{prop}

\prf  We will use the criterion of  Theorem \ref{thm:approxequiv} for ycleptness: a $\bimer$ is yclept iff there is an essentially DL map into a finite metric space such that two models are equivalent if the images under the map are closed.

An essentially DL formula $R$ taking values in a finite metric space $X$ is called {\em quasi-invariant} if whenever $M \ee N$, we have $d(R^M(a),R^N(a))< 1$ for all $a$.   If $t$ is a completion of $T$, $R:M^n \to X$ is {\em $t$-quasi-invariant} if whenever $M \ee N$, with $M, N \models t$ we have $d(R^M(a),R^N(a))< 1$ for all $a$.  Equivalently $d(R^M(a),R^N(a))\leq v_X$, where $v_X$ is the greatest possible distance  below $1$ of
pairs of  points of  $X$.    If $R$ is $t$-quasi-invariant, then by compactness the $t$-quasi-invariance is due
to a single sentence $\si \in t$.   We can modify $R_X$ with no prejudice to $t$ by redefining it as $R$ if $\si$ holds, and some constant element
of $X$ if $\si$ fails. This is  quasi-invariant.  

For each completion $t$, by assumption and  Theorem \ref{thm:approxequiv}, there exists a $t$ quasi-invariant mapping that witnesses this.
And from the argument above, there exists a quasi-invariant, not just $t$-quasi-invariant,  $R_t:M^n \to X_t$  such that if 
$M,N \models t$, then $M \ee N$ iff $d_{X_t}(R_t^M(a),R_t^N(a))< 1$ for all $a \in M^n$.   It follows that for $M,N \models T$ on the same universe,
$M \ee N$ iff $d_{X(t)}(R_t^M(a),R_t^N(a))< 1$ for all $a$ and for every $t$.   Using the definability of $\ee$, and compactness, there must exist 
a finite set $t_1,\ldots,t_k$ such that 
$M \ee N$ iff $d_{X(t)}(R_t^M(a),R_t^N(a))< 1$ for all $a$ and for every $t_i,i=1,...,k$. At this point  let $X = \prod_i X_{t_i}$, with  distance $\sup_i d_{ X_{t_i} }(x(i),y(i) )$.  Using dummy variables assume all $n_t = n$ for some $n$.  Define
$R(a,b) = (R_i(a,b): i=1,\ldots,k)$.   Then $M  \ee N$ iff $d_{X}(R^M(a),R^N(a))< 1$ for all $a$. Thus using Theorem \ref{thm:approxequiv} in the other direction, we are done.

\eprf

\section{An yclept, non-ydlept  $\bimer$, and chaotic dynamics.} \label{sec:notydlept}

We exhibit here an yclept $\bimer$ which is not ydlept.  Thus  Theorem \ref{thm:approxequiv} cannot be pushed further to say that a $\bimer$ which is  yclept must be ydlept, even for superstable or weakly minimal theories.

There is a tension between ensuring that a yclept is not ydlept 
and maintaining the definability (as opposed to only the infinitary definability) of the $\mer$.
To assert that a CL formula has the same interpretation in two models, one must express that 
at any input they take the same real value; this seems to require   infinitely many formulas, 
asserting equality digit-by-digit as it were.   We overcome this using a phenomenon of  chaotic dynamics:   two orbits cannot permanently stay near each other.    Thus knowing  the approximate location of a point over all time determines the point precisely.

This example is moreover superstable, weakly minimal.   We will define a DL theory $T_d$,  and
an equivalence relation $\ee$ on models of $T_d$ based on 
preserving a CL reduct.

Let $\G$ be a countable group, acting by homeomorphisms on a connected compact  Hausdorff space $W$
 in such a way  that for some 
open set $U_1$, letting $U_0=W \m cl(U_1)$, we have: 

\medskip

 (*) for any $u \neq v \in W$, for some $\g \in \G$ we have $\g u \in U_1$ and $\g v \in U_0$.  

 \medskip

 This implies that for any $w \in W$, $\G w$ is dense in $W$.
 For instance, $PGL_2(\Qq)$ acting on the circle $\Pp^1(\Rr)$ by fractional linear transformations has this property, 
 for any nonempty and non-dense open set $U$,
since $PGL_2(\Rr)$ is $2$-transitive on   $\Pp^1(\Rr)$.   
Later in the proof, we will need additional requirements
 about $(\G,W)$, still true for $(PGL_2(\Qq),\Pp^1(\Rr))$, and we 
focus on a specific $U_1$.
In terms of $(\Gamma,W)$ we will require:

\medskip

(**)  For any nonempty open   $U_1$,  and  $U'' \subset W\m cl(U_0)$ and any $u \neq v \in W$ there exists $\g \in \G$ with 
$\g u \in U'', \g v \in U_1$.

\medskip

 Note that (*) and (**) persist for  subgroups  $\G$ of  $PGL_2(\Qq)$ that are dense in $PGL_2(\Rr)$.  One can take $\G$ finitely generated.
 In this case $T_d$ will admit a finite language. 

As for $U_1$,
identifying $\Pp^1(\Rr)$ with $\Rr \union \{\infty\}$, we will
later specify that $U_1$ is the interval $(-2,2)$, and 
we will also fix $U''$ to be the interval $(-1/2,1/2)$.   These specific
choices will play no role in the first two claims below, defining the reduct, and
through it the equivalence relation.
But they will be helpful in showing that our equivalence relation is a $\bimer$.

Let $L_0$ be the language with a unary function symbol for each $\g \in \G$, that we will denote also by $\g$.  
The universal theory $T_0$ consists of  the axioms for a  free $\G$-action; 
it becomes complete once we add that the universe is nonempty.

$L_d= L_0 \union \{Q\}$ is the expansion of $L_0$ by a  unary predicate $Q$.   We also write $Q_1$ for $Q$ and
$Q_0$ for the complement of $Q$.  Further we will write $Q(a)=1$ if $a \in Q_1$ and $Q(a)=0$ if $a \in Q_0$.

 The universal axioms $T_{d,\forall}$ of $T_d$ state that 
 \[ \meet_{i=1}^m {\g_i} \inv Q_{\nu_i} = \emptyset \]
  whenever $\g_1,\ldots,\g_m \in \G$, $\nu_1,\ldots,\nu_m \in \{0,1\}$,
   and $\meet_{i=1}^m \g_i \inv cl(U_{\nu_i}) = \emptyset$.
   
Note that $\gamma_i$ in the intersection refers to a unary function symbol in the
language, while in the condition $\gamma_i$ refers to the transformation.
Thus the axioms assert a close relationship between the unary predicate
$Q_1$ in a model of $T_d$ and the open subset $U_1$ of $W$.

To see that $T_{d,\forall}$ is consistent, 
pick any point $a \notin \union_{\g \in \G} \g (cl(U_0) \setminus U_0)$.  Let $M=\G a$ be the $\G$-orbit of $a$;
so $M \meet cl(U_0) \setminus U_0 = \emptyset$.
Define $Q_1^M =M \meet U_1$.   Then $Q_0^M = M \meet cl(U_0) = M \meet U_0$.   It is clear that $M$ is a model of 
$T_{d,\forall}$.

Towards defining our CL reduct, we start with the following observation:
 
 \begin{claim}  \label{clm:runique} Let $M \models T_d$ and $a \in M$.  Then 
  \[ R(a):= \meet \{ \g \inv cl(U_i):   \g \in \G, i=Q(\g(a))  \} \]
is a subset of $W$ with a unique point.   
\end{claim}

\prf  The axioms of $T_d$  ensure the finite intersection property for the above sets.  Since $W$ is compact,
$\meet \{ \g \inv cl(U_i):   \g \in \G, i=Q(\g(x)) \}  \neq \emptyset$.   If $u \neq v $ are two distinct points
of $W$ then by (*), for some $\g \in \G$ we have $\g u \in U_1$ and $\g v \in U_0$.  Say $Q(\g(a)) = 0$.
Then $R(a) \subset \g \inv cl(U_1)$  while $v \in \g \inv (U_0)$ so  $v \notin R(a)$.  Likewise if $Q(\g(a)) =1$
then $u \notin R(a)$.  Hence $R(a)$ cannot contain two distinct points, so $|R(a)|=1$.  
\eprf

We denote the point whose existence is shown by the claim as $r(a)$.
The behavior of $r(a)$ in terms of $U_1$ closely mirrors that of $\a$ in terms
of $Q_1$, and similarly for $\gamma(a)$ with $\gamma \in \G$.  For example if $r(a) \in U_1$ then $a \in Q_1$, and if $a \in Q_1$ then
$r(a) \in cl(U_1)$.
The function $r$, along with the $L_0$ structure,  will be our CL reduct $L_c$.

\begin{claim}   The function $r(a)$ depends only on $tp(a)$, and is a continuous function of $tp(a)$.  
\end{claim}
\prf  The first statement is obvious.  To prove continuity, let $O$ be an open neighborhood of $r(a)$.  
Then $ \meet \{ \g \inv cl(U_i):   \g \in \G, i=Q(\g(a)) \}  \subset O$, so by compactness of $W$, some 
finite intersection  $ \meet_{\g \in F_0} \{ \g \inv cl(U_i):   \g \in \G, i=Q(\g(a))\}  \subset O$.  But this means
that the truth values of finitely many atomic relations $Q(\g_i(x))$ force $r(x) $ to be in $O$, i.e.
$r \inv(O)$ contains a clopen neighborhood of $tp(a)$.    \eprf

At this point we may view $r$ as a CL definable relation on $M$.     
 Define ${\mathcal G}$ to be the groupoid of $L_c$-isomorphisms between models of $T_d$,  i.e. bijections $g$ preserving $L_0$  and $r$, but not necessarily $Q$.
  $\ee$ is the corresponding equivalence relation.  
Since $r$ is CL definable, this is a CL reduct.
 
The next thing to show is that $\mathcal{G}$ is a $\bigp$, or alternatively $\ee$ is   a $\bimer$.
In fact, we get universal definability. But this is to be expected, from Theorem \ref{thm:approxequiv}.

  Let $\g_1$ be the element $x \mapsto x+1$.   Note that if $\g_1 \inv a,a, \g_1 a \in cl(U_1)$ then $a \in U_1$.

\begin{claim}  \label{clm:eedefinable} Let $\ee$ be the equivalence relation over $T_d$ corresponding to preserving the $L_0$
structure and $r$.
Then  $M \ee M'$ iff $M$ and $M'$ agree on the $L_0$ structure and
\[ \diamond \ \ \  (\forall x \in M) ( M \models Q_1 ({(\g_1)} \inv x) \wedge Q_1(x) \wedge Q_1 (\g_1 x)) \rightarrow M' \models  Q_1(x) )\]

Thus in particular $\ee$ is a $\bimer$.
\end{claim}
\prf  For the ``only if'' direction, we assume that there is 
$x_0 \in M$ such that $Q_1$ holds of $x_0$, $\g_1(x_0)$ and $ {(\g_1}) \inv(x_0)$, but
in  $M'$ $Q_1(x_0)$ fails to hold. The hypothesis concerning $M$
tell us  that in $M$ we must have $r(x_0) \in cl(U_1)$, $\g_1(r(x_0)) \in cl(U_1)$,
$\g_1 \inv (r(x_0)) \in cl(U_1)$, and keeping in mind $U_1=(-2,2)$ we have $r(x_0) \in (-1,1)$.
On the other hand, in $M'$ we cannot have $r(x_0) \in cl(U_1)$. So clearly
$r$ is not preserved by the equivalence relation.

  We prove the ``if'' direction by contrapositive.  Suppose 
$M$ and $M'$ agree on their domains and the $L_0$ structure, let $r$ denote
the function $r$ in $M$ and $r'$ the function in $M'$.
Assume $r(a) \neq r'(a)$ for some $a \in M$.  By (**), there exists $\g \in G$ with 
$\g(r'(a)) \in U_0$ and $ \g(r(a)) \in U''$.
But we can check that $r$ ``commutes modulo starring'' with $\g$ in any  model.  

Thus we have
$r'(\g a) ) = \g r'(a) \in U_0$ and $r(\g a) \in U''$.  Let  $b  = \g a$; then  
$r'(b) \in U_0$, and $r(b) \in U''=(-1/2,1/2)$.  
So applying the  definition of $r$ in $M$, we get
$M \models {{(\g_1)}} \inv b , b, \g_1 b \in Q_1$.  By $\diamond$, $M' \models Q_1(b)$, contradicting 
$r'(b) \in U_0$.  \eprf

We now need to show that $\groupoid$ is not given by a DL reduct.   We begin by 
studying a CL theory $T_c$ that will turn out to be a reduct of $T_d$.

 Let $L_c$   expand $L_0$ by    a unary predicate $r$ in the sense of CL, interpreted as a map into $W$.   The universal axioms,
denoted $T_{c,\forall}$, state that $T_0$ holds and $r$ is a $\G$-homomorphism,
 i.e. $r(\g x) = \g r(x)$, referring to the action above on $W$.

\begin{claim}  \label{clm:uniqueext} There is
  a (necessarily unique) theory $T_c$ eliminating quantifiers with universal part axiomatized by $T_{c, \forall}$.
\end{claim}
   
\prf  Let $T_c$ be $T_{c,\forall}$ along with axioms:
\medskip

 ($\clubsuit$)   For each $w \in W$ there exist infinitely many $\G$-orbits $O$ of $N$
 with elements $a$ such that $r(a)=w$.  

 \medskip

Any model of $T_{c,\forall}$ clearly extends to a model of $T_c$; so the universal part of $T_c$ is precisely $T_{c,\forall}$.

To see that $T_{c}$ is complete, suppose $T',T''$ are two completions.  Let $M,M'$ be models of $T',T''$ respectively, of cardinality $2^{\aleph_0}$,
and such that for any $w \in W$ there exist $2^{\aleph_0}$ distinct elements $a$ of $M$ (resp. $M'$) with $r(a)=w$.  It is easy to find such models by compactness.

We argue that $M \cong M'$.  Let $(w_i: i \in I)$ be a set of representatives for the $\G$-orbits on $W$.   For each $i$, let $m_{i,j}: j < 2^{\aleph_0}$ be
a maximal subset of $M$ consisting of elements in distinct orbits, such that $r(m_{i,j}) = w_i$.  Then any element of $M$ has the form $\gamma m_{i,j}$
for some $i \in I, j < 2^{\aleph_0}$ and some $\gamma \in \Gamma$, necessarily unique.  In the same way, we find $m'_{i,j} \in M'$ with $r(m'_{i,j}) = w_i$.   Let $f$ map
$\gamma (m_{i,j})$ to $\gamma (m'_{i,j})$.  This is well-defined and an $L_c$-isomorphism. So $M \cong M'$.

 Hence $T'=T''$. So $T_c$ is complete.
 
In a similar way, we can construct an isomorphism between countable substructures of $M,M'$ extends to an isomorphism $M \to M'$.  This holds also when $M=M'$;   so the quantifier-free type of a tuple from $M$ determines the complete type.   Thus $T_c$ admits quantifier-elimination.  
  
 \eprf
  
   By QE for $T_{c}$, or more directly by the argument for Claim \ref{clm:uniqueext}, complete $1$-types have the form $r(x)=w$ for some $w \in W$.  A partial type can be viewed as a closed set of complete types, 
   and hence has the form  $r \inv(C)= \{x: r(x) \in C \}$ for some closed $C \subset W$.

\begin{claim} \label{clm:definable} Every DL definable set $X$ of $T_c$ in $n$ variables is already $L_0$-definable
\end{claim}
\prf  Take $n=1$.  As $X$ is $\bigwedge$-definable, it is a partial type. By the comment above the claim, it has the form $r \inv(C)$ for some closed $C \subset W$.
As the complement of $X$ is also $\bigwedge$-definable, it also has this form, so that $X = r \inv(C')$ for some open $C' \subset W$.
By surjectivity of $r$, in sufficiently saturated models, we have $C'=r(X)=C$ so $C$ is  clopen.  Since $W$ is connected, $C=\emptyset$ or $C=W$
and so  there are no unary definable sets other than $\emptyset$ and the universe.     For $n>1$, let $P$ be the $L_0$-partial
$n$-type asserting that $x_1,\ldots,x_n$ lies in distinct orbits.   Then an $n$-type extending $P$ is determined by the values of $r(x_1),\ldots,r(x_n)$.
A partial $n$ type extending $P$ is the intersection of $P$ with the pullback of a closed subset of $W^n$.  Since $W^n$ is connected, the same argument shows
that if $X$ is a definable set in $n$ variables, then either $X$ or the complement $X$ is disjoint from $P$; say the former.  Then $X$ implies a finite disjunction
of terms $\gamma x_i = x_j$.   So $X = X_1 \union \ldots \union X_k$ where $X_i$ includes such a term.  By induction each $X_i$ is $L_0$-definable and hence so is $X$. \eprf

\begin{claim} Let $M \models T_d$, and let $M_0$ be the reduct to $L_0$.   Then 
$(M_0,r)$ is a model of $T_c$.  
\end{claim}
\begin{proof}
For this we may assume $M$ is saturated.   Recall the set function $R$ from Claim \ref{clm:runique}.
Clearly, $R(\g x) =     \meet \{ \b \inv cl(U_i):   \b \in \G, i=Q(\b(\g a)) \} =  
    \meet \{ \b \inv cl(U_i):   \b \g \in \G, i=Q(\b(\g a))  \} = \meet \{ ({\a \g} \inv) \inv cl(U_i): \a \in \G, i=Q(\a a) \} 
    = \g R(x)$.  
 So   $r(\g x) = \g r(x)$. 
  Thus $(M_0,r) \models  T_{c,\forall}$.   We also have to show  ($\clubsuit$): 
  
  \medskip
  
  For
any $w \in W$ there exist, in infinitely many $\G$-orbits, $a \in M$ with $r(a)=w$ 

\medskip

Define  
an $L_d$-structure on the orbit $\G w$ by setting $Q^{ w} = \G w \meet U_1$.  If 
$\meet_{i=1}^m \g_i \inv cl(U_{\nu_i}) = \emptyset$, then $\meet_{i=1}^m {\g_i} \inv Q_{\nu_i} = \emptyset$.
This is because  $Q_1 \subset cl(U_1)$ and $Q_0 \subset cl(U_0)$.   Thus $(\G w, Q^w) \models  T_{d,\forall}$.
And so there exist (many disjoint)  $L_d$-embeddings $j:(\G w, Q^w) \to M$; and we have $w \in R(j(w))$ so $w = r(j(w))$. 
Letting $a = j(w)$, we have proven our claim.
\end{proof}

We have seen in Claim \ref{clm:definable}  that  DL reducts of $T_c$ must be  $L_0$ definable.
Hence $(M_0,r)$ has no
DL reduct bigger than $L_0$.   
So $L_0$ is the maximal DL reduct respected by $\ee$.  Certainly $r$ is not $L_0$-definable, since $T_0$ has a unique $1$-type: thus any $L_0$-definable map into the reals  is constant.
Hence $\ee$ is not ydlept.

  \section{Higher ydlept equivalence relations that are yclept} \label{sec:yclepinthigherydlept}
  
In the previous section we considered the interaction of $\bimer$s with yclept $\mer$s, and showed that this does not imply  that the $\mer$ is ydlept.
In this section we look at the intersection of higher ydlepts with yclept. In contrast to the previous section, we show that this does imply that the $\mer$ is ydlept.

  Let $T$ be a theory,  for simplicity in a countable language.  Let $D$  be one of the sorts or a finite product of sorts of $T$.
 
  Recall that  in a CL theory with trivial metric, a subset $X$ of $D$ is called definable if  both $X$ and $D \setminus X$ are $\bigwedge$-definable.  Equivalently,
  there exists a formula $\phi$ valued in $\{0,1\}$ with $\phi \inv (1)=X$.
  
  \begin{prop}\label{prop:higherydleptsubyclept}  Let $\ee$ be an yclept $\bimer$ on $T$,   associated to the CL reduct $\mathcal{R}$.  Let $\mathcal{D}=\{D_c: c \in Q\}$ be a $T$-definable family of definable subsets.
  Assume $\mathcal{D}$ is $\ee$-invariant, i.e. invariant for the groupoid corresponding to $\ee$.   Then  $\mathcal{D}$
can be expressed as $\mathcal{D} =\{D'_d: d \in Q'\}$ with $D'$ and $Q'$ both
  contained in the reduct $\mathcal{R}$.   
  \end{prop}

Before proving this, we give a corollary:

\begin{cor}\label{cor:ycleptand2ydlept}  A $\mer$ that is both $2$-ydlept and yclept is ydlept.
  \end{cor}

\prf Let $\{D_c: c \in Q\}$ witness that the $\mer$ $\ee$ is $2$-ydlept, and $\mathcal{R}$ witness
that $\ee$ is yclept.  Then the proposition implies
that the family can be reparameterized as $\{D'_d: d \in Q'\}$ with $D',Q'$ in $\mathcal{R}$.
Then $D',Q'$ are invariant under $\ee$, so $\ee$ is the ydlept equivalence relation generated
by  $D', Q'$.
\eprf

Note that if we add in Corollary \ref{cor:ycleptand2ydlept} that the $\mer$ is definable, we can conclude that it is a definable ydlept, and hence is given by finitely many first order formulas.

As a step towards proving the proposition, we start with:

\begin{claim}
Let  $\mathcal{R}$ and $\mathcal{D}$ be as in the statement of Proposition \ref{prop:higherydleptsubyclept}. In any countable resplendent model $M$ of $T$, each individual $D_c$ is definable with parameters in $M$ from $\mathcal{R}$ alone. We recall that for CL
 this means that both the set and its complement are $\bigwedge$-definable using symbols from $\mathcal{R}$ with parameters.
\end{claim}
\prf 
Let $\groupoid$ be the groupoid correspoing to $\ee$.
Any $Aut(M|\mathcal{R})$-conjugate of $D_c$ lies in $\mathcal{D}(M)$.   This is because $\mathcal{D}$ is   $\groupoid$-invariant, and 
$\groupoid= Aut(M|\mathcal{R}$ since $\groupoid$ is generated by $\mathcal{R}$.  Thus there are only countably many  $Aut(M|\mathcal{R})$-conjugates of $D_c$.   We can apply  
Makkai's parameterized  version of Beth's theorem \cite{makkaibeth}, which can be seen to apply also to the CL case.  It yields that there exists a finite $d$ from $M$ such that $Aut(M|\mathcal{R}, d)$ fixes $D_c$.  
So there are no $a \in D_c, a' \notin D_c$ with the same type  over $d$. That is, the following $L(c,d)$-partial type is inconsistent:
 
\[\{ \phi(x,d)=\phi(x',d ) \wedge x \in D_c \wedge x' \notin D_c  : \phi \in \mathcal{R} \} \]

So for some finite number of $\phi_i$, which we can put together into a single
$\phi$ valued in $[0,1]^n$,  the $\phi^M$-images of $D_c$ and of $D \setminus D_c$ 
are disjoint closed subsets, whose union of course covers $\phi^M(D)$.  Thus the image of $D$ in $[0,1]^n$ is disconnected,
and there exists a continuous map $f: [0,1]^n \to \{0,1\}$
such that $f(\phi(x))=1$ if $x \in D_c$ and $f(\phi(x))=0$ if $x \notin D_c$.   \eprf

We now continue the proof of the proposition. 
\prf
Let us say that $\phi(x,d)$ is a \emph{gap formula} if it takes values in $[0,1] \setminus [1/5,4/5]$.  
Consider the following set of formulas:

\[ 
(\forall d) (\exists x)( 1/5 \leq \phi(x,d) \leq 4/5) \vee (\exists x)(x \notin D_u \iff   \phi(x,d) \leq 1/5 )
\]
 Here the formula includes $u$ as a free variable, and $D_u$ refers to $\mathcal D$ at variable $u$.  $\phi$ ranges over $\mathcal{R}$.
For each $\phi$ the formula  asserts that $\phi$ is either not a gap formula or  does not define $D_c$.
The conjunction  is inconsistent.
 So   there are finitely many formulas $\phi_i \in \mathcal{R}$, $i=1,\ldots,\ell$, valued in $[0,1]$, such that for each $c$,
and for some $d$, $\phi_i(x,d)$ is a gap formula and defines $D_c$.  

Let $\psi_i (i=1,\ldots,\ell)$ be an essentially DL formula --  corresponding to a finite set of pairs of $L$ formulas
and real values --
such that $|\phi_i - \psi_i| < 1/5 $.   Then for each $c$,
for some $i \leq l$, for some $d$:
\[
\psi_i(x,d)  < 2/5  \iff x \in D_c
\]

For each $i \leq l$, consider $\theta_i$ defined as   
\begin{align*}
\{ d ~| ~ \exists Y \in \mathcal{D} ~ (\psi_i(x,d)<2/5 \iff x \in Y) \}
\end{align*}
This is $L$-definable and $\ee$-invariant, so it is in $\rR$.
 
  Thus the equivalence relation given by our original reduct $\mathcal{D}$ is the same
  as the equivalence relation preserving the family of sets
  $\{x ~ \mid ~ \psi_i(x,d) \}$ where $d$ ranges over all values such that $\theta_i(d)$.

   \eprf

\begin{cor} \label{cor:ycleptcaphigherydleptimpliesydlept}  A $\mer$ $\ee$ that is yclept and $k$-ydlept is ydlept.  \end{cor}
\prf  Using Corollary \ref{cor:ycleptand2ydlept}
and induction, along the same lines as the proof of \propref{prop:stable-n} from  \lemref{lem:invariantdef}.  \eprf

Proposition \ref{prop:higherydleptsubyclept} dealt with a higher ydlept preserved by
a higher ydlept.
We can also get a result on yclepts preserved by a higher ydlept, but this time only for maximal ones:

\begin{prop} \label{prop:maxycleptofhigher} The maximal yclept of a $k$-ydlept
must be ydlept.
\end{prop}
\begin{proof}
We prove the case $k=2$, leaving the general case to the reader.
Thus $\ee$ is specified by  DL formulas $\phi_i(\vec x, \vec o)$ in the Shelahization, where $\vec o$ is over the imaginary sort. Then $\ee$ represents preservation of the imaginary sort and the corresponding function, and also the  $\phi_i$. Let $R'$ denote this reduct.

Thee maximal CL reduct consissts of CL formulas $\rho(\vec x)$ where $\vec x$ are variables in the language $L$ of $T$. We know $\rho$ must be preserved by $\ee$.
We claim that each such $\rho$ must be definable in CL over $R'$. If not
there are two models $M_1$ and $M_2$ with the same reduct to these formulas and sorts,
disagreeing on $\rho$. But this contradicts the fact that $\rho$ is invariant under $\ee$.

Thus $\rho$ is a limit of essentially DL formulas $\rho_i$ definable over $R'$. Each of these formulas is specified by a DL-definable partition $\rho_{i,j}(\vec x)$.
Since they are definable over $R'$, each $\rho_{i,j}(\vec x)$ is invariant under $\ee$.
But then each $\rho_{i,j}(\vec x)$ is in the maximal CL reduct.
We see that $\ee$ is in fact equivalent to preserving $\rho_{i,j}$, hence is  ydlept.
\end{proof}

\begin{rem} \label{rem:coupledycleptand}
Note that we state and prove these results for the ``completely-coupled'' case. But the same arguments hold if we talk about a ``yclept $\mer$ over coupled sorts $S_0$''.
 Recall that a reduct defined using formulas having sorts $S_0$ gives a $\mer$ over $S_0$, or over any set of sorts $S'_0$ containing $S_0$: we are free to couple more.
 Thus modifying Corollary \ref{cor:ycleptcaphigherydleptimpliesydlept} we can conclude that if a $\mer$ is both yclept over $S'_0$ and $n$-ydlept over the same $S'_0$, then it
 is ydlept over $S'_0$. We cannot conclude that if it is yclept over $S'_0$ and $n$-ydlept over a different $S''_0$ then it is ydlept over $S''_0$.
 \end{rem}
 
\section{Intrinsically stably embedded theories and higher ydlepts} \label{sec:stable}

In this section we characterize theories for which every higher ydlept $\mer$ is ydlept.   It turns out to be an interesting extension of the class of stable theories.   We conjecture that every o-minimal structure has this property, and prove it for DLO.

Recall our convention that ``definable'' means $0$-definable, i.e. by a formula of $L$.

\begin{defn}\label{def:stablyembedded} 
A set of sorts $\sorts$ in a multi-sorted theory $T'$ is \emph{stably embedded in $T'$} if in every model $M$ of $T'$,
any $M$-definable relation on the sorts $\sorts$ is definable with parameters 
from $\sorts(M)$.
\end{defn}

Equivalently, by a routine compactness argument every $T'$-definable family of definable relations on $\sorts$ is equal, as a family, to a $T'$-definable family with parameter variables in $\sorts$:

\medskip

For every $n$, for every $\phi(x,y)$ with $x$ an $n$-tuple of variables with sorts in $\sorts$,  $\phi$ in $T'$ -- that is,
a $T'$  definable family of subsets
of $n$-tuples of elements of $\sorts$ -- there is a formula $\delta(x, x')$, where $x$  are the same variables as in $\phi$ and $x'$ is a tuple of variables with sorts in $\sorts$ $\delta$ a formula in $T'$, such that, in all models of $T'$, there is $x'_0$  with $\forall y ~ \exists x'_0 ~ \forall x ~ \phi(x,y) \leftrightarrow \delta(x,x'_0)$.

For background on stable embeddedness, equivalent definitions, and extensions to partial types, 
see \cite[appendix B] {diffgal}, \cite[Appendix]{chatz}, \cite{pillayembedded}.

 \begin{defn}\label{def:istem} A theory $T$ with sorts $\sorts$ is {\em intrinsically stably embedded} (istem) if $\sorts$ is stably embedded within any reduct of $T^{eq}$.  \end{defn}

In particular, 
 a one-sorted theory $T$ is  istem  if its only sort is stably embedded within any reduct of $T^{eq}$.
It is easy to see that a stable theory is istem: if $T$ is stable, so is $T^{eq}$,
and hence any reduct of $T^{eq}$.

 \begin{lem} \label{lem:stable-2} Let $T$ be an istem theory with one sort $X$.
 If the model equivalence relation $\ee$ is $2$-ydlept then $\ee$ is ydlept.   
 \end{lem}

For simplicity we prove the lemma for a definable $\ee$ that is $2$-ydlept; the proof will easily be seen to work
for general $2$-ydlept nodels equivalence relations.
Thus let the equivalence relation $\ee$ be given by   a $0$-definable family of definable subsets 
of $X$. 
So it has the form $\{D_b:  b \in Y \}$, where $D \subset X \times Y$ and
$D_b=\{a \in X: D(a,b)\}$, and $D$ is definable, hence $Y \subseteq X^n$ is definable.

The following lemma states that in an istem theory we can re-parameterize a definable family in a canonical way:

\begin{lem} \label{lem:invariantdef}
Let $T$ be an istem theory in one sort $X$.  Let 
   $D \subset X \times Y$ be definable in $T$ 
${\mathcal{F}}=\{D_b:  b \in Y \}$, $D_b=\{a \in X: D(a,b)\}$. 
  Then there exists a definable relation $R \subset X \times Z$ ($Z$ a definable subset of $X^n$) such that 
\begin{enumerate}
\item ${\mathcal{F}}=\{R_c:  c \in Z \}$; 
\item Any permutation leaving  $X$ and ${\mathcal{F}}$ invariant also leaves $R$ invariant.
\end{enumerate}
\end{lem}

\prf 
We consider the imaginary sort related to the definable family $D$ itself: this will have a sort $S'$  for the equivalence classes of the
equivalence relation ``coding the  same member of the definable family'', and the corresponding relationship $F$ relating equivalence classes in $S'$ to elements
$x$ of the definable family. Thus we can consider $D$ as a family, but now indexed by $S'$.  It is still a definable
family of subsets of $X$ in the reduct $R'$ consisting of $S'$ and $F$. Since  $X$ is intrinsically stably embedded, we can obtain a $\delta'(x, \vec y)$ definable in the reduct $R'$,
such that for every member $R_c$ of the family there is $\vec a_0$ such that $\delta'(x, \vec a_0)$ is $R_c$. $\delta'$ can be converted
to a formula in the language of $T$. But since it is definable in the reduct $R'$, it is invariant under any permutation that preserves the equivalence classes.

.

\eprf

\prf[Proof of \lemref{lem:stable-2}] By Lemma \ref{lem:special} we can assume $\ee$ is given by
a $2$-set: that is, a definable family on some   $X^n$.
Apply \lemref{lem:invariantdef} to get a definable family  $R$ the generates the same $2$-set, but
any permutation leaving  $X$ and ${\mathcal{F}}$ invariant also leaves $R$ invariant.
It is easy to to that $X$ and $R$ witness
that $\ee$ is ydlept. 
\eprf

We now extend Lemma \ref{lem:stable-2} from $2$-ydlept to $n$-ydlept for $n>2$.

 \begin{prop} \label{prop:stable-n} Let $T$ be an istem theory.
If the model equivalence relation $\ee$ is $n$-ydlept for some
$n$ then $\ee$ is ydlept.

In particular, this is true for a stable theory.
 \end{prop}

 \prf  
 Assume inductively that every $n$-ydlept on a stable theory is ydlept.
Let $\ee$ be $(n+1)$-ydlept.  By definition there exists  a definable family $F$ of definable sets,  such that $\ee$ arises from a well-behaved equivalence relation $\ee_1$ of $T_F$, where $\ee_1$ is $n$-ydlept.   By induction, $\ee_1$ is ydlept.   Hence $\ee$ is $2$-ydlept.  By \lemref{lem:stable-2}, $\ee$ is ydlept.

\eprf

We show that the istem assumption is sharp; thus $T$ is istem iff every higher ydlept is ydlept.  In fact:

\begin{prop} \label{prop:istemandydleptcollapse}  $T$ is istem iff  every $2$-ydlept $\mer$ for $T$ is ydlept.    \end{prop}
\prf
Assume every 2-ydlept $\mer$ for $T$ is ydlept.    We will show that $X$ is istem.  The other direction has already been proven.

We may  assume the language of $T$ is finite:  if $T$ is not istem, this is witnessed by a finite sublanguage; while the hypothesis on $2$-ydlepts is by definition preserved under reducts.  
We have to show that $X$ is stably embedded within any reduct $\rR$ of $T^{eq}$.  We may take this reduct to include all sorts of $T^{eq}$, but not necessarily all relations among them.   Let $\rR$ be such a reduct, and $T_{\rR}$ the restriction of $T$ to $\rR$. So $T_{\rR}$ still has a unique sort $X$, but the reduct contains only some of the relations of $T$.

Let $Y$ and $R \subset Y \times X^n$ be $0$-definable in $T^{eq}$.  For $y \in Y$, let $R_y \subset X_n$ be defined by 
$x \in R_y \iff R(x,y)$.  Let $R_Y:= \{R_y: y \in Y\}$.   We have to show $\clubsuit$:

\medskip

Each $R_y$ is  definable in $\rR$ with parameters in $X$.

\medskip

Note that if $M \models T$ then we can canonically define $M^{eq}$ and hence $R_Y$: it is a $2$-subset of $P(X(M))$ that is defined in terms of the structure $M$ alone.

We have a $2$-ydlept  $\bimer$ $\ee$ on $Mod(T)$:   $M \ee M'$ iff $X(M)=X(M')$ and 
$R_Y^{M} = R_Y^{M'}$.
By assumption, $\ee$ is ydlept.  So there exist an $\rR$-definable relation $S$ on $X^{n}$ such that
$M \ee M'$ iff $X(M)=X(M')$ and $S^{M} = S^{M'}$.   
 Here $X(M)=X(M')$ simply says they have the same universe.
 Recall the Chang-Makkai theorem  \cite{changbeth,makkaibeth}, in the following ``pseudo-elementary'' form:    

\medskip

Let $T$ be a theory in a countable language $L$, and let $T'$ be an expansion to a bigger countable language, including a formula $R_0$, $T=T' \upharpoonright  L$.  If for any countable model $M$ of $T$, $\{R_0(M'):  M' \mbox{ expands } M, M' \models T' \}$ is countable, then each $R_0(M')$ is definable with parameters in $M$.   

\medskip

\def\fF{{\mathcal F}}
 
 This will be applied to  the family $\fF$  of subsets $Z$ of $X^n$ such that for
some expansion of $M_{\rR}$ to a model $M$ of $T$, $M^{eq} \models (\exists y \in Y) (R_y = Z)$.  That is, $Z \in R_Y ^M$.

\begin{claim}  Let $M_{\rR} $ be a countable model of $T_{\rR}$.   Then $\fF(M_{\rR})$ is countable. \end{claim}

\prf
To prove the claim, let $M_1$ be an expansion of $M_{\rR}$ to a model of $T$.  We will show that $\fF \subset R_Y(M_1)$; this of course
implies that $\fF$ is countable.  
Let $Z \in \fF$.   Then $Z \in R_Y(M)$ for some   expansion $M$ of $M_{\rR}$ to $T$.  Since
$S$ is $\rR$-definable, we have $S^M = S^{M'}$.  Thus $M \ee M'$.  So $R_Y^{M_1}=R_Y^{M}$.
Hence $Z \in R_Y^{M_1}$, as promised.
\eprf
Given the claim, we can apply the Chang-Makkai theorem, concluding that each element of $\fF$ is parameterically definable in $\rR$. Thus we have proven $\clubsuit$, which in turn proves Proposition \ref{prop:istemandydleptcollapse}.
\eprf

 We note a corollary of Proposition \ref{prop:stable-n}. A class of structures for a fixed language is said to be stable if the (incomplete) theory of the class is stable; equivalently for each $\phi(\vec x, \vec y)$ there is a maximum
$n$ such that in some model in the class there are $\vec a_1 \ldots \vec a_n$ $\vec b_1 \ldots \vec b_n$ with $\phi(\vec a_i, \vec b_j)$ iff $i <j$. For finite
graphs, the notion relates to several notions of sparsity \cite{adleradler}. 

\begin{cor} \label{cor:finstable} If $\ee$ is an $n$-ydlept over a stable theory $T$, then $\ee$ is ydlept.  In particular if $\mathcal C$ be a class of structures that is stable, and $\ee$ is $n$-ydlept, then $\ee$ is an ydlept equivalence relation when restricted to structures
in $\mathcal C$.
\end{cor}

\begin{proof}
 Any completion $T^*$ of $T$ is stable in the usual sense. Thus by Proposition \ref{ydlept-conts}, $\ee$ is ydlept on  models of the completion, hence on models of $T$.
\end{proof}

\begin{conjecture} \label{conj:ominimalimpliesistem}  All o-minimal structures are istem.
\end{conjecture}

A proof of the conjecture would imply a positive answer to the following:

\begin{question} \label{quest:ominimalimpliesistemcollapse} Does every o-minimal theory have the property that $2$-ydlepts are ydlept?
\end{question}

As a small step towards proving the conjecture, we show that DLO is istem.  The proof will go via a somewhat wider class of one-sorted structures.

\begin{defn}  We say that a structure is Cameron Thomas-minimal, or \emph{CT-minimal} for short, if every reduct, except pure equality, interprets the whole upon adding constants.\end{defn}

Restated, this says that every nontrivial reduct of $T$  is a reduct by constants, in the sense of \defref{def:reductbyconstants}.
From Cameron's classification of reducts of DLO \cite{reductsdlo} it follows that DLO is CT-minimal.   
Since DLO is NIP,  the result for DLO will follow from the following:

\begin{prop}  \label{prop:nipctistem} Any NIP  CT-minimal structure is istem.\end{prop}

\prf  We will assume $T$ has a single sort $D$ for convenience.
Let $M$ be the given structure; we may assume it is $\aleph_0$-saturated.  
Let $Q$ be an imaginary sort, and $M^-$ a reduct with sorts $(D,Q)$.     We have to show that $D$ is stably embedded in $M^-$.

We can work over a somewhat saturated model $M$, so that if $F$ is a family of definable sets, not definable in pure equality, then some element of the family is not definable in pure equality.

If every family of definable subsets of any $D^n$ in $M^-$ is definable in pure equality, then of course $D$ is stably embedded.    A \emph{stably $1$-embedded} sort is one which satisfies the definition of stably embedded where we only consider definable families of subsets, not definable families of collections of tuples: that is, we only consider $n=1$.
We will use Pillay's theorem \cite{pillayembedded}: for an NIP theory   $D$ is stably embedded iff it is \emph{stably $1$-embedded}.
 
Using the theorem, we can take $n=1$.   Let $F$ be a nontrivial family of definable subsets of $D$. 
Then by the $\aleph_0$-saturation of $M$  some element $F_c$ of $F$ is infinite and co-infinite.  

As a first case, assume that pure equality is induced on $D$ in $M^-$.
Let $m$ be any integer. Since $F_c$ is infinite and co-infinite, there is $A$ a subset of $D(M)$ with $|A \meet F_c |= | A \setminus F_c| =m$.    Since in the reduct, all $2m$-types
of distinct elements of $D$ are equal, for any subset $A'$ of $A$ 
of size $m$, there exists $c'$ such that for $a \in A$ we have $a \in F_{c'}$ iff $a \in A'$.   Since $m$ was arbitrary, we have shattered an arbitrarily large set $A$, which contradicts NIP.

This leaves the case where the induced structure on $D$ from $M^-$ is not pure equality.  By CT-minimality, after adding
constants $(d_1,d_2,\ldots)$  from $D$, the full original structure on $D$ is interpretable in $M^-$.   In this case, any $M^-$ definable relation on $D$ with parameters from $M^-$  is certainly definable with parameters from $D$ in the original structure,  hence in $M^-[d_1,d_2,\ldots]$, and hence in $M^-$.  
\eprf

The following example shows that this phenomenon does not extend to general NIP theories.
Later examples interpretable in the random graph will
show the same for simple theories.

\begin{example}  \label{ex:nipbutneedafamily}
Let $T$ be the  $\aleph_0$-categorical theory of a dense linear ordering $M$ with a distinguished
dense co-dense subset $P$.  Then $T$ is NIP,   indeed distal and dp-minimal \cite{simondistalpairs}.
We view $T$ as $2$-sorted,
with one sort $P$ and one sort $Q$= complement of $P$.
 The ordering on $P \union Q$ is thus  viewed as three relations, namely  $<_P = < \meet (P^2)$,
 $<_Q = < \meet (Q^2)$ and $<_{PQ} = < \meet  (P \times Q)$.
Let $\groupoid$
be the group of permutations $\si$ of $M$ that respect $Q$, and hence $P$,
 such that $\si \upharpoonright Q$  lifts to an automorphism of $M$.
An equivalent definition is that the equivalence relation preserves the family
of definable sets $R_p: p \in P$, where  $R_p= \{x \in Q$ and $x>p\}$.
In particular $\groupoid$ is $2$-ydlept. 

Let us now see that it is not yclept.
Note first that any finite order-preserving partial map from $Q$ to $Q$ extends to an element of $\groupoid$.

We next claim that, from the above, we can conclude that
the maximal $\groupoid$-invariant reduct of $M$ is  the structure consisting
of the sort $P$ with no additional structure, as well as the sort $Q$ with
the restricted linear order. We denote this as $(P,Q,<_Q)$. That is, 
we claim a formula $\kappa(x_1 \ldots x_m)$ is preserved by $\groupoid$ if and only
it is  definable in $(P, Q, <_Q)$.
The direction
from right to left follows since the definition of $\groupoid$ is
in terms of $R_p$ above. So it suffices to show that if
$\kappa(x_1 \ldots x_m)$ is not definable in $(P,Q, <_Q)$
then it is not  preserved by $\groupoid$. Note that  $(P,Q, <_Q)$ is
also $\aleph_0$-categorical, so there are finitely many $m$-types over the empty
set. Since $\kappa$ is not definable over the empty set, there must be a model with
some  $\vec c$ and $\vec c'$ satisfying the same type in the restricted language,
but with $\vec c$ and $\vec c'$ disagreeing on $\kappa$. Considering the
finite partial automorphism of $(P, Q, <_Q)$ mapping sending $\vec c$ to $\vec c'$, and applying the above, we get an element of 
$\groupoid$ that does not preserve
$\kappa$.

Thus we know that if $\groupoid$ were yclept, it must be definable in the language above.
And since the automorphism group of
$(P,Q, <_Q)$ is much bigger than $\groupoid$,  $\groupoid$ cannot be definable from $(P, Q, <_Q)$.
So $\groupoid$, while being $2$-ydlept, is not  ydlept, and thus is not  yclept.

We note that this theory does admit other groupoids that are  not $n$-ydlept.
Even DLO has these: see the example in Subsection \ref{ex:deftop}.
\end{example}

\begin{question} Does there exist a theory that is ``$\forall$-2-ydlept -- every $\bimer$ is $2$-ydlept -- but the theory is
not $\forall$-yclept? 
\end{question}

See \propref{prop:smallclasses} for a setting where 
$2$-ydleptness occurs naturally.

 \section{$\mer$s of low quantifier complexity} \label{sec:aecomp}

 We can classify $\bimer$s  by the  quantifier complexity of their defining $\lpluslprime$ sentences.  We regard all $L$-formulas, and their primed copies, as 
 having quantifier complexity $0$; equivalently, we assume quantifier-elimination
in the base theory.
With this convention assumed, we employ the usual terminology for prenex classes of first order formulas: universal, $\Pi_2$ etc.
 An equivalence relation on models of $T$ defined by   $\Sigma_n$   (respectively $\Pi_n$) formulas of $\lpluslprime$
 is said to be a $\Sigma_n$ (resp. $\Pi_n$) $\inftybimer$. And similarly for  a $\bimer$ if just finitely many such formulas are used.
We will now explore the relationship of $\mer$s restricted by prefix classes with each other and with the classes defined earlier.

 \begin{thm} \label{thm:ae} Let $\ee$ be a $\Pi_2$  $\inftybimer$ for  $T$.
 Then $\ee$ is yclept. 
 \end{thm}

From this and   Theorem \ref{thm:approxequiv} we obtain:
\begin{cor} \label{cor:ae}  Let $\ee$ be a $\Pi_2$  $\inftybimer$   for  $T$.  Then $\ee$ is 
 actually a $\Pi_1$ $\inftybimer$.
 \end{cor}

Note that 
we cannot conclude $\ee$ is an ydlept $\mer$: see Section \ref{sec:notydlept}.

We now begin steps toward the proof of Theorem \ref{thm:ae}.
We can assume, using Proposition \ref{yclept-conts},
 that $T$ is complete.
 We also assume, as we may, that $T$ admits QE.   Let $\groupoid$ be the associated groupoid.   Recall that $\groupoid(M,M')$, 
the set of $\groupoid$-morphisms from $M$ to $M'$,
 is the set of interpretations of the isomorphism symbol that make $M,M',\si$ a model of a certain $\lpluslprime$-theory $Th(\groupoid)$.
Recall also Definition \ref{def:bimax}, the theory $\bitheorymax$ of the maximal reduct.
 \begin{lem}\label{lem:ae2} Assume  $\groupoid$ is   $\Pi_2$.   Let $(M,M',g)$ be an existentially closed model of $\bitheorymax$.
Then $g \in \groupoid(M,M')$.
 \end{lem}
 
\prf  By \lemref{lem:ae1},    $(M,M',g) \models Th(\groupoid)_\forall$.  In general, for a theory $T$ with universal part $T_{\forall}$, a structure $A \models T_\forall$ iff $A$ extends to a model of $T$; thus  $(M,M',g)$ extends to a model $(N,N',g')$ of $\groupoid$.   
Since $(M,M',g)$ is existentially closed,
 every $\Pi_2$ sentence true in $(N,N',g')$ is true in $(M,M',g)$.  But by assumption $\groupoid$ is $\Pi_2$ axiomatizable.
 So $(M,M',g) \in \groupoid$, i.e. $g \in \groupoid(M,M')$.
\eprf

  \begin{lem}\label{lem:ae3}   Let $(M,M',g)$ be a resplendent countable  model of $\bitheorymax$.    Then there exists $M'' \models T$ 
  and bijections $h: M \to M''$, $j:M' \to M''$ such that $(M,M'',h)$ and $(M',M'',j)$ are existentially closed models of $\bitheorymax$, and   $jg=h$.
 \end{lem}
 \prf   Let $M''$ be a saturated copy of $M$, i.e. a saturated model of the same cardinality.  We will construct
 $h,j$ by an induction of length $|M|$.    At a given stage $i$ we have small substructures $A=A_i$ of $M$, $A'=A'_i$ of $M'$
 with $g(A)=A'$, and $\maxred{\groupoid}$-embeddings $h_i: A \to M'', j_i: A' \to M''$ with $j_ig=h_i$.    We will extend $h_i$
 in some stages, and $j_i$ in other stages; then we will immediately define $j_i$ (respectively $h_i$)
 by the formula $j_ig=h_i$.    
 
 At stages $0,4,8,\ldots,\omega, \omega+4, \ldots$, we will aim to make $h$
 onto, i.e. we find the least element of $M''$ not   hit by a previous $h_{i'}$, and extend $h_i$ by mapping to it 
 an element of $M$ realizing the appropriate $\maxred{\groupoid}$-type over $A$.   Similarly, we can make $j$ onto at stages
 $1,5,9,\omega+1,\ldots$.  At stages $2,6,\ldots$ we seek to make $(M,M'',h)$ existentially closed.
 We are handed by bookkeeping two $T_{A_i}$-types $p(x),q(y)$ with the same restriction to $L^{\mathcal G}$ over $A_i$.
Then we simply realize $p$ by $a$ from $M$, we realize $q$ by  $b$ from $M''$ and we extend $h_i$ by mapping
$a \mapsto b$.  Likewise at stages $3,7,\ldots$ we ensure $(M',M'',j)$ is existentially closed.
  \eprf

 We now  complete the proof of  Theorem \ref{thm:ae}. Fix $(M,M',g)$ a saturated (or resplendent) model of $Th(\groupoid)$
and let $h,j, M''$ be as in the conclusion of 
 Lemma \ref{lem:ae3}. 
By Lemma \ref{lem:ae2} we have $h \in \groupoid(M,M'')$ and $j \in \groupoid(M',M'')$,
 hence $g \in \groupoid(M,M')$.  Thus for any saturated $(M,M',g)$  in the groupoid,
 $\maxred{\groupoid}$ satisfies $g \in \groupoid(M,M')$. 
 Hence the same holds for arbitrary  $(M,M',g)$ in $\maxred{\groupoid}$.   So 
 if $\si$ is a $\maxred{\groupoid}$-automorphism
then $\si \in \groupoid(M,M')$. Thus we have shown that the groupoid
corresponding to the maximal CL reduct is the groupoid itself, and thus the groupoid is yclept.
 This proves \thmref{thm:ae}.

 \begin{rem}  \thmref{thm:ae} is sharp, in the sense that  a $\Sigma_2$ $\bimer$ need not by yclept; see \secref{Quantifier complexity}.
 \end{rem}
 
 \begin{cor} \label{cor:unary} 
 Consider a theory $T$ in a language 
 with a single sort $V$, 
 with unary predicates $P_i$ and functions $F_j$ only, such that $T \vdash ``F_i$ is bijective''
 for each $i$. Then every $\bimer$ for $T$  is yclept.

Let $L$ be a language with a single sort $V$, a binary function symbol $+$,
predicates $P_i \subseteq V^{n_i}$ and function symbols $P_i: V^{m_i} \to V$.  
Let $T$ be a theory asserting that $+$ is an abelian group operation,
 and that each $P_i$ is a subgroup (of some $V^n$) and each $F_j$ is a homomorphism (from some $V^n$ into $V$).   Then every $\bimer$ for $T$ is yclept. 
 \end{cor}
 
 \prf   We prove the first part. The theory $\lpluslprime$ still satisfies the same description. That is, it consists of  unary predicates and invertible unary functions.  Hence it admits QE to the level of Boolean combinations of existential formulas.  In fact, as is easy to check, it admits QE as soon as we name the inverse function of each invertible function in the language. 
In particular, this is the case for sentences, and  \thmref{thm:ae}  applies.
  So any  $\bimer$ is yclept.
  
The second part is proven similarly, using QE for abelian structures, to Boolean combinations of positive primitive formulas \cite{fisher}.
\eprf

\section{Smoothness and ycleptness} \label{sec:smooth}

We  give a more robust characterization of yclept $\mer$s. relating it to the notion of smoothness
on Borel equivalence relations. We then provide two applications, one to definable equivalence relations in the usual sense (Subsection \ref{subsec:product}) and the other to consequences of cardinality restrictions on $\mer$s
(Section \ref{sec:smallfew}).

In this section we \emph{consider a consistent theory $T$ in a countable language with no finite models}.   We will only consider countable models,
indeed models with universe $\Omega=\Nn$.   This suffices of course for determining a  $\mer$ if the language is countable.

Fix a universe $\Omega$ and consider a  $\bimer$  $\ee$ on $\modelsof_\Omega(T)$.
By Morley-izing $T$, noting that the set of models of $T_\forall$ on   $\Omega$ is naturally a Polish space.
In the Morleyization, the axioms are $\forall \exists$, thus the set of models of $T$ forms a $G_\delta$
subset, we can also view  $\modelsof_\Omega(T)$ as a Polish space.    $\ee$ induces a Borel equivalence relation on $\modelsof_\Omega(T)$.   So one can ask if it is smooth in the sense of descriptive set theory, 
i.e. admits complete   invariants in a standard Borel space.   Recall that $\ee$ is smooth if there exists a Borel map $\phi$ on $Mod_\Omega(T)$ into a Polish space, such that $xEy$ iff $\phi(x)=\phi(y)$.  This depends only on the Borel structure of $\modelsof_\Omega(T) $, not on the Polish space structure.
See \cite{dh} for another use of this notion in model theory.

\begin{thm}\label{thm:smooth}      Let $\ee$ be a $\inftybimer$.  Then $\ee$ is yclept iff $\ee$ is smooth.
\end{thm}

\prf   We will show that these are also equivalent to:  $\groupoid(M,M)$ is closed in $Sym(\Omega)$, for any resplendent countable $M$.

Here we take the usual group topology on $Sym(\Omega)$; it is induced from the product topology  on $(^\Omega \Omega)^2$ via the map $g \mapsto (g,g\inv)$. $\Omega$ itself is taken to have the discrete topology.   We note that  for the functions in the groupoid, this is the same as the pointwise convergence topology.
 In general, a sequence of invertible functions may converge to a non-invertible one.  But if $g_n,g$ lie in the groupoid and 
 if $g_n$ pointwise converges to $g$,    for each $a$, we must have $g_n(a)=g(a)$ for large $n$, since the topology on $\Omega$ is by definition discrete. Then letting $b=g^{-1}(a)$ we have $g_n(b)= g(b)=a$ for large $n$,   so $g_n^{-1} (a) = b = g^{-1}(a)$ for large $n$. This equivalence allows us to apply the characterizations in
 Proposition  \ref{prop:maxred} and its corollaries.

 Assume $\ee$ is yclept.
We note that  ycleptness implies that the group $G(\Om,\Om)$ is closed  in $Sym(\Omega)$: this follows from
Corollary \ref{cor:closedyclept}. 
Thus $\ee$ is smooth: the invariant function is just the ``reduct'' function $\modelsof_\Omega(T) \to \modelsof_\Omega(T')$.  
We leave to the reader the choice
of Borel codes for $\modelsof(T)$ and $\modelsof(T')$ and the verification that the restriction map is Borel.

Now we prove the direction from closed to yclept. Assume $\groupoid(M,M)$ is closed for all $M$, or just for $\groupoid$-resplendent $M$.   
Then $\ee$ is yclept by \corref{cor:closedyclept}.

 Finally, we prove the direction from smooth to closed.  Assume $\ee$ is smooth.  Let $M$ be a homogeneous and resplendent model of $T$, meaning that for any tuple
   $a \in M^n$ there exist $b \neq b' \in M$ with $tp(b/a)=tp(b'/a)$. Therefore $\dcl(a) \neq M$.
  We will show that $\groupoid(M,M)$ is closed in $Sym(M)$.
  
  The group  $H:=\groupoid(M,M)$ is a Borel subgroup of $Sym(\Omega)$. In fact, it is at a finite level on the Borel hierarchy, since $\ee$, or the associated groupoid, is definable.
  Let $\bH$ be the closure of $H$ in $Sym(\Omega)$.    
  Note that $\bH$ is a perfect Polish space:  if $\bH$ has an isolated point, then by translating this point to the identity element $1$ we see
 that $1$ is isolated.  Thus there exists some tuple $a \in M^n$ such that the only element $h$ of $\bH$ with $h(a)=a$ is the identity.
 In particular this is true for $Aut(M)$ since $Aut(M) \subset \groupoid(M , M)=H$.   But that contradicts the assumption that 
 $M$ is homogeneous and resplendent. 
  
   If $H$ is non-meager in $\bH$, then by  \cite{pettis}
  $H=H H \inv$ contains an open neighborhood of the identity of $\bH$. Hence $H$ is open. Since $\bH$  is a topological group, this implies that it is closed in $\bH$. Thus $H$ must equal $\bH$, and thus $H$ is closed in $Sym(\Omega)$, and we are done in this case. 
  
  So consider the case where $H$ is meager in $\bH$. Then by   \cite[Theorems 3.4.3 and 3.4.5]{bk},  
     the  right coset equivalence relation of  $H$, namely $Hx=Hy$,  is not smooth. 
     However, note that for $g \in Sym(\Om)$, we have
   $M' \ee M''$ then $gM' \ee gM''$, since $g$ is an isomorphism from the $\lpluslprime$-structure $(M',M'')$ to $(gM',gM'')$, and $\ee$ is bipartite
definable, and thus  invariant under
   isomorphisms of the $\lpluslprime$-structures.   Thus for $g_1,g_2 \in \Sym(\Om)$, we see that $g_1,g_2$ are right $H$-conjugate iff $g_1 =g_2 h$ for some $h \in H$
   iff $g_2 \inv g_1 \in H$ iff $g_2 \inv g_1 M \ee M$ iff $g_1M \ee g_2 M$ iff $e(g_1M) = e(g_2M)$ where $e$ is a Borel function giving a complete invariant for $\ee$.
   This shows that  the right $H$-coset equivalence relation is smooth on $Sym(\Om)$ and in particular on $\bH$, a contradiction.
    \eprf

   \begin{rem}  While smoothness is defined with respect to the class of all models of $T$ with universe $\Om$, the proof shows that smoothness
   on the (Borel) subset of homogeneous models of $T$ with universe $\Om$ implies smoothness in full; likewise when restricted only to copies of the saturated model, if there is one, etc.
\end{rem}

We can use the characterization to argue that
 ycleptness is preserved under finite-index refinements.  There should be a common generalization of \corref{cor:finiteindex} and \lemref{fewclasses}, but we will not pursue it here.

\begin{cor} \label{cor:finiteindex} Let $\ee_0$ be an yclept  $\mer$.
Let $\ee$ be a $\bimer$ refining $\ee_0$, such that each $\ee_0$-class splits into finitely many  $\ee$ classes. 
Then $\ee$   is yclept.  \end{cor}

\prf  Since $\ee_0$ is yclept, for a suitable (CL) sublanguage $L_0$ of (the CL language generated by) $L$ we have $M \ee_0 N$  iff $M \upharpoonright L_0 = N \upharpoonright L_0$. 
Let $T_0 = T | L_0$. Note that $T_0$ is a CL theory.
Let $\Mod_{r} (T_0)$ be the set of sufficiently saturated   models of $T_0$ with universe $\Nn$ (or a finite initial segment). We defer the details of the saturation assumption, but refer to the analogs of resplendency for CL theories. 
Let $Z_{\geq k}$ be 
the set of $M_0 \in Mod_{r}(T_0)$ that expand to  $k$ pairwise $\ee$-inequivalent models of $T$.  We observe that $Z_{\geq k}$ is Borel. This follows from the fact the existence of such an expansion reduces to a consistency statement:
if an elementary extension $M^*$ has such an expansion, so does $M$ by the saturation assumption.

Let $Z_k^0=Z_{\geq k} \m Z_{\geq k+1}$.
Then $Z_k^0$ is Borel.   By the sentence just before the beginning of Section \secref{sec:defs},
we can find a Borel section $b$ from $Z_k^0$ to $k$-tuples of expansions, so $b(M_0) = (b_1(M_0),\ldots,b_k(M_0))$ 
witnesses the $k$ inequivalent classes.   Let $k(M)$ be the unique $k$ such that $M \upharpoonright L_0 \in Z_k$. Such a $k$ exists by the assumption on $\ee$.

Further define a Borel function $c: \Mod_r(T) \to \Nn$, with $c(M)=c$ iff
 $M ~ \ee  ~ b_{c}(M \upharpoonright L_0 )$. Thus $c(M)\leq k(M)$.   Now for $M,N$ resplendent models of $T$ 
 we have $M \ee N$ iff $k(M)=k(N)$, $c(M)=c(N)$, and
$M \upharpoonright L_0 =N \upharpoonright L_0$.  This shows that $\ee$ is smooth on resplendent models of $T$, and hence yclept.

\eprf

We now restate, in a slightly  more general form, what we have learned thus far about $\bimer$s.

\begin{thm} \label{thm:aerelative} 
Let $T$ be a theory in language $L$ And  $\ee$ be a $\bimer$ for $T$.
The following are equivalent:
\begin{itemize}
\item $\ee$ can be defined by a $\Pi_2$ sentence.
\item $\ee$ can be defined by a $\Pi_1$ sentence.
\item  For any $M \models T$, $Aut_{\ee}(M)$ is closed in $Sym(M)$.
\item  $\ee$ is  yclept
\item  
$\ee$ is closed in the set of pairs of models
of $T$ with the same underlying universe for the coupled sorts,
for the topology where $\{M : M \models \phi(\vec a))\}$ is basic open.
\item  (equivalent for $L$ countable) 
$\ee$ is smooth
as a Borel equivalence relation on models of $T$. 
\end{itemize}
\end{thm}

We can see from the proof of Theorem \ref{thm:smooth} that for $T$ complete, quantification over all models 
can be replaced by considering any resplendent model.

We now apply these characterizations to get a  result classifying the $\bimer$s for a given theory.
 
 \begin{defn} 
Say that a  theory $T$ is:
\begin{itemize}
\item  \emph{$\forall$-ydlept} if every  $\bimer$
 is given by a DL reduct  
\item \emph{$\forall$-yclept} if every  $\bimer$ 
is given
by a CL reduct
\end{itemize}
\end{defn}

Corollary \ref{cor:unary} shows that certain theories are $\forall$-yclept.
Let us show that the property is closed under reducts.

\begin{prop}\label{prop:reducts}     Let $L$ be a language, $L_0$ a sublanguage with the same sorts, $T_0=T|L_0$. If  a $T_0$ $\inftybimer$ is yclept when viewed as
a $T$-equivalence relation, then it is yclept as a $T_0$-equivalence relation.   

Thus if $T$ is $\forall$-yclept,  any reduct $T_0$ is also $\forall$-yclept.

\end{prop}

When  $T$ is complete and one-sorted, the contrapositive of the first statement reads:    if an $\inftybimer$ is not yclept, then expanding the theory within the same sort cannot make it yclept.

\prf  Let $\modelsof_{rs} (T_0)$ be the set of recursively saturated models of $T_0$ with universe $\Omega=\Nn$.  
Following a variation on Henkin's construction, any recursively saturated model $M_0$ of $T_0$ can be extended in a deterministic way to a model of $T$;
this is the theorem that recursively saturated models are resplendent.  This construction can be represented within $Th(\Nn,M_0,T)$. Indeed
it can even be represented in  Peano arithmetic with predicates for $M_0$ and for $T$, induction axioms allowing these predicates.   Hence 
 there is a   $Th(\Nn,M_0,T)$-definable expansion of $M_0$ to a model of $T$. 
This gives
 a Borel section $s:\modelsof_{rs}(T_0) \to \modelsof_\Omega(T)$.  Any $\ee_0$ on $T_0$ also defines an equivalence relation $\ee$ on
$\modelsof_\Omega(T)$, namely $M \ee M' $ iff $M_0 \ee_0 M'_0$ where $M_0,M'_0$ are the reducts to $L_0$.
If $\ee$ is smooth, there is a Borel   function $e$ with $e(M)=e(M')$ iff $M \ee M'$. But then $e \circ s$ is a Borel invariant function for  $\ee_0$, and hence  $\ee_0$ seen as an equivalence
on $\modelsof_{\Omega} (T_0)$ is smooth too. We now have an $\inftybimer$ which is smooth over resplendent models, and this implies that it is yclept, as noted after Theorem \ref{thm:aerelative}. 
\eprf

\propref{prop:reducts}  can also be proved using a continuous logic version of \cite{makkaibeth}, 
along the same lines as the ydlept case below (\lemref{prop:dreducts}).   We still give the descriptive set-theoretic proof 
in order to demonstrate the method.   

For readers interested in descriptive set theory we point out a generalization, and a question.

\begin{prop}   Let $L_0 \leq L$, $T_0=T|L_0$, let $\ee_0$ be a $\mer$ of $T_0$ inducing a $\mer$ $\ee$ of $T$.   Then $\ee_0 | Mod_{\Nn}(T_0)$
is Borel-reducible to $\ee | Mod_{\Nn}(T)$ as  Borel equivalence relations.    
\end{prop}

To see this, find a Borel map from pairs $(M_0,N_0) \in \Mod_{\Nn}(T_0)$ to pairs $(M,N) \in \Mod_{\Nn}(T)$ with $(M_0,N_0) \prec (M,N)$.   Then $(M_0,N_0) \in \ee_0$ iff $(M,N) \in \ee)$.

As a special case, if $\ee$ is smooth then so is $\ee_0$.   

We note in this connection that an $n$-ydlept $\mer$ is $\Pi_{2n}$-definable. 
\begin{question}
Is it  true that a $\Pi_{2n}$-definable, $m$-ydlept $\mer$ is $n$-ydlept?  
 \end{question}

Here is the ydlept version of \propref{prop:reducts}.

\begin{prop}\label{prop:dreducts}  Let $L$ be a language, $L_0$ a sublanguage with the same sorts. Let $T$ be a theory over $L$ and $T_0=T|L_0$.  If $\ee_0$ is an $\inftybimer$ for $T_0$  and is ydlept when viewed as
a $T$-equivalence relation, then it is ydlept. 

Thus if $T$ is $\forall$-ydlept, any reduct $T_0$ is also $\forall$-ydlept. \end{prop}

\prf  Let $\ee$ denote $\ee_0$ lifted to $T$. 

For simplicity, assume $\ee$  is ydlept via the reduct to one relation $R$.  
Then for any models $M_0, N_0$ of $T_0$, for any expansion $M$ of $M_0$ and  $N$ of $N_0$ to $T$, if  $M_0 \ee_0 N_0$, we must have $R(M)=R(N)$. 
Assume $M_0$ admits an expansion to $T$, and take  $N_0:=M_0$.  Then  there  is a unique expansion of $M_0$ to $L_0 + \{R\}$ such that $T| L_0 + \{R\}$ holds.  

By Beth's theorem, $R$ is explicitly definable in $T$, by some $L_0$ formula $\phi$.   Thus for any $M_0, N_0$ which have an extension to $T$ we  have $M \ee_0 N$ iff  $\phi(M)=\phi(N)$.

In general, $\ee_0$ lifted to $T$ is ydlept via $R_1,R_2,\ldots$, and the same proof shows each of these is explicitly definable by an $L_0$ formula $\phi_i$,
and $M \ee_0 N$ iff  $\phi_i(M)=\phi_i(N)$ for each $i$.

\eprf 

We believe that the analogous result can be proven for $k$-ydlepts, making use of Chang-Makkai, but have not verified this.

We can use the results on reducts to get more examples of $\forall$-yclept theories.
 
\begin{cor}    \label{cor:forallydleptequivalencerelation} The theory of an equivalence relation with infinitely many classes, all of size $n$, is $\forall$-ydlept.
\end{cor}
\prf It is a reduct of the theory of $\Zz/n \Zz$-actions, which is unary and hence $\forall$-yclept by \corref{cor:unary}.   By  Proposition \ref{prop:dreducts}, 
or by  \corref{prop:reducts} along with smallness,   it is ydlept.
\eprf

\ssec{Smoothness of definable equivalence relations on products} \label{subsec:product}
Mostowski and Feferman-Vaught studied the theories of  
infinite products of structures.  
We apply the smoothness characterization to obtain results on definable equivalence relations \emph{within} a model of such a theory: that is, the usual notion of ``definable equivalence relation''.

  Let $I$ be an index set, and $M_i$   $L$-structures.   Let $M=\prod_{i \in I} M_i$.      We will view $\prod_{i \in I} M_i$ with the language described
  in \cite{dh}, 2.1.  Namely there is an additional sort for the Boolean algebra $P(I)$, and for each formula $\phi$ of $L$ in $n$ variables,
  a function symbol $[\phi]$ in the same $n$ variables, taking values in $B$.  The intended interpretation is $[\phi](a)=\{i: M_i \models \phi(a(i))\} \in B$.
  
  When the $M_i$ and $M$ are countable, $\prod_{i \in I} M_i$ has a natural separable Polish space structure, and in particular a Borel structure.

\begin{thm}   Let $I$ be a countable index set, and $M_i$ countable $L$-structures.   Let $M=\prod_{i \in I} M_i$.    
 Then any definable equivalence relation on $M^n$ is smooth.  
\end{thm}

The rest of this subsection will provide the proof of theorem, which will make use of the smoothness characterization for $\mer$s.

\prf  Along with the product we will consider a dual structure, $M_*:=\coprod_{i \in I} M_i$.  As a set it is the  $I$-indexed disjoint sum; each basic relation
$R$   is interpreted as $\coprod R(M_i)$; and we add another sort $I$ and a  map  $\varpi:  M_* \to I$ whose fibers are the $M_i$.  The theory $T_*$ describes
this structure, asserting that there are no relations across different $\varpi$-fibers, and that each $\varpi$-fiber is a model of $T$.  Let $T_*[n]$ be the expansion
by symbols for $n$ functions $s$ from $I$ to $M_*$.  The theory will express that  $\varpi \circ s (i) = i$ for all $I$, and similarly for any element of the product $M= \prod_{i \in I} M_i$.

 Note that any  function from $I$ to $M_*$
 is entirely determined by the image $s(I) \subset M_*$. So we can treat $s$ as a unary predicate.   
 
\begin{claim}  Under the correspondence defined immediately above,  any  definable subset  of $M^n$   maps to an elementary subset of the space of expansions 
$(M_*,u_1,\ldots,u_n)$ of $M_*$.  More precisely it is the set of models expanding $M_*$  of $T_*[n]$ along with a certain 
  Boolean combination of universal and existential sentences.  
\end{claim}

\prf  
By the Feferman-Vaught quantifier-elimination (as described e.g. in Section 2.1 of \cite{dh}),   a formula of $\prod_{i \in I} M_i$ takes the form  $P([\phi_1],...,[\phi_n])$,
where $[\phi]$ is the Boolean truth value of $\phi$, a subset of $I$,  and $P$ is some formula of atomic Boolean algebras.    
Keeping in mind quantifier elimination in Boolean Algebras, we can assume that this is a Boolean
combination of assertions that some $[\psi]$ contains at most $m$ atoms, or at least $m$ atoms.   And this can be expressed by the universal sentence:
\[(\forall t_0 \in I) \ldots (\forall t_m \in I) \bigvee_{i\leq m}  \ \psi(  u_1(t_i),\ldots,u_n(t_i))   \]
and the obvious existential sentence for the ``at least'' case.
Of course if we treat $u_i$ as a unary predicate on $M_*$ rather than as a function, we can either universally or existentially introduce variables $z_{i,j}$ for $u_j(t_i)$, 
with $u_j(z_{i,j})$ and  $\varpi(z_{i,j})=t_i$, so it does not change the quantifier complexity.
\eprf

Now assume $E(x_1,\ldots,x_n; y_1,\ldots,y_n)$ is a definable equivalence relation on $M^n$.    By the claim, $E$ induces a $\bimer$ $\ee$ on models
of $T_*[n]$, refining the ydlept equivalence relation associated to the reduct $T_*$ of $T_*[n]$.  And furthermore $\ee$ is defined by a Boolean combination of universal (and existential) sentences, in particular $\Pi_2$.   Hence by Theorem \ref{thm:aerelative} $\ee$ is yclept. 
 By the smoothness characterization  -- the easier direction of Theorem \ref{thm:aerelative} -- it defines a smooth equivalence relation $\ee'$
on the space of expansions of $M_*$ to $T_*[n]$.   But we have a bi-continuous identification of this space of expansions with $M^n$, identifying $E$ and $\ee'$.
Thus $E$ is smooth on $M^n$. \eprf

\section{$\mer$s with small classes and few classes} \label{sec:smallfew}

 Let $T$ be a theory in a countable language.  We say that a   $\mer$ $\ee$ on $T$  has {\em small classes} if 
 for any countable set $\Om$, no    equivalence class of $\ee$ on  $\modelsof_{\Om}(T)$ has    size $2^{\aleph_0}$.    A full characterization of this class of $\mer$s   would be very interesting.  
Here we apply the results on smoothness to show that it is essentially a subclass of the $2$-ydlept class, fully determine it within the $1$-ydlepts, and give some intriguing examples 
of non-yclept $\mer$s with small classes.   We also characterize 
 $\mer$s with few classes.       
 
 \begin{lem}\label{lem:smallfinite}   Let  $T$ be a complete theory, and let $\ee$ be a  $\bimer$ with small classes.   Then $T$ essentially has a finite language, i.e. for some finite $L_1 \subset L$, every relation is $T$-equivalent to an $L_1$-formula.
 \end{lem}
 
 \prf   Take $L_0$ a large enough finite sublanguage of $L$ mentioning any symbol used, either primed or unprimed, in $\lequalslprime$ to define $\ee$.  Then the $\ee$-class of $M$ is determined by $M \upharpoonright L_0$.  Let $T_0=T \upharpoonright L_0$.
 Then for any model $M_0$ of $T_0$, all expansions of $M_0$ to a model of $T$ are $\ee$-equivalent. 
 Suppose $T$ does not have an essentially finite language, i.e. no such $L_1$ exists.  Then there are $T$-definable elements $c_1,c_2,\ldots$
 such that no $c_n$ is definable in $L_0(c_1,\ldots,c_{n-1})$.  Now it is easy to build a tree $c_\eta:  \eta \in 2^{<\omega}$ such that for each $\eta$
 $tp_{M_0}(c_1,\ldots,c_n) = tp_{M_0}(c_{\eta \upharpoonright 0 },c_{\eta  \upharpoonright 1},\ldots,c_{\eta\upharpoonright n})$ and for any $\eta$, $c_{\eta ~ 0} \neq c_{\eta ~1  }$.  
  Any branch
 gives a different expansion of $M_0$ to a model $M$, a contradiction.  \eprf

 \begin{prop} \label{prop:smallclasses} Let $T$ be a   theory in a finite language $L$.  Let $\ee$ be a $\bimer$ on $\modelsof_{\Om}(T)$
 with small classes.  Then $\ee$ is $2$-ydlept.   \end{prop}

 Note that we do not assume completeness here. If $T$ is complete, the finite language assumption is not necessary,  by \lemref{lem:smallfinite}.

\prf  We may take the language to consist of a single relation symbol $R$.  Then in any countable model $M=(\Omega,R)$ of $T$, there are only countably
many solutions $R'$ to $R \ee R'$.  By the Chang-Makkai theorem \cite{changbeth,makkaibeth}, applied to the  $\lpluslprime$ sentence defining $\ee$,
it follows that there is a formula $\phi(\vec x, \vec y)$
over  $R$ such that in any pair $M \ee M'$  with the same domain, there is $\vec p$ in the common domain
such that $R'=\{\vec x ~ | ~  \phi(\vec x, \vec p)\}$. 
Conversely we may assume that for any parameter $\vec p$,
$\{x: \phi(x,\vec p)\} \ee R$.
This is because we can modify $\phi$, setting it equal to $R(x)$ whenever the above fails.

The formula $\phi$ thus witnesses that the relation is a $2$-ydlept.
\eprf

Although the above tells us quite a bit about $\bimer$ with small classes, it turns out that those that ydlept are even more restricted.

\begin{defn} \label{def:reductbyconstants}  Let $L_0 \subset L$ be languages and let $T$ be an $L$ theory. Let $T_0$ be the restriction of $T$ to $L_0$.  We say that $L_0$ is a
 {\em reduct by  constants} of $T$  if $T$ is equivalent to an expansion by constants of $T_0^{eq}$.
 In other words there exist  sorts $S_i$ of $T_0^{eq}$ and  $T$- $0$-definable elements $c_i$ of $S_i$, such that $T$ is obtained from $T_0$ by adjoining the $c_i$; i.e. for every relation $R$ of $T$ there exists a formula $\phi(x,u)$ of $T_0$
with $T \models R(x) \iff \phi(x,c)$, with $c$ a tuple of these $c_i$.   

We say $L_0$ is a  {\em reduct by a constant} of $T$ if this is true with a single sort $S$ and $c \in S$.
    \end{defn}

\begin{prop} \label{prop:small1}     Let $T$ be a   theory, and let   $\ee$ be a definable ydlept with small classes.   Let $L_\ee$ consist
of predicates for the
first-order definable sets of $L$ that are invariant by the groupoid of $\ee$, and let $T_\ee=T \upharpoonright L_\ee$.    Then $L_\ee$ is a reduct by a constant of $T$.  
\end{prop}
\prf       By \lemref{lem:smallfinite},  we may   assume that $L$ is generated by a single relation $R$.  
 Since any countable model $M_0$ of $T^\groupoid$ expands in only countably
 many ways to a model $M=(M_0,R^M)$ of $T$, by \cite{makkaibeth},   $R^M$ is  $M_0$-definable with parameters.  
 By passing to an imaginary sort of $T_{\ee}^{eq}$, we may take $R$ to be defined by $\phi(x,b)$, where $b$ is a canonical parameter for $\phi$,
 and $\phi$ belongs to $L^\groupoid$, as defined in Definition \ref{def:languagetheoryofgroupoid}.  Then $b$ is a $T^\ee$-definable constant, and adding $b$ to $L_\ee$ gives us all of $L$ back.   
\eprf  

The same statement and proof hold in the yclept case, given that \cite{makkaibeth} extends to CL theories.
Based on \propref{prop:small1}   one can also describe the incomplete case; we leave the precise formulation to the reader.

We now give some examples of model equivalences  with small classes.
We will show that we cannot strengthen \propref{prop:smallclasses}  to conclude that such an equivalence relation is ydlept, or even yclept: see
 Example  \ref{ex:smallnotyclept} and \ref{ex:fincofinitemodeven}.

Recall from Proposition \ref{prop:smallclasses} that a $\bimer$ with small classes is  $2$-ydlept, for a finite language, or a  complete theory.
We give two examples 
of  $\bimer$s over a theory in a finite language  that have small classes, but which are not yclept.  The first involves an incomplete theory,  but in view of \remref{yclept-conts}, there is a completion in which
it is still non-yclept.  


\begin{example} \label{ex:smallnotyclept} 
Let $T_0$ be an extension of ZFC, or of a fragment of ZFC large enough to make sense of forcing, in  $L_0 = \{\e\}$.   
Let $L=L_0 \union \{Q\}$ with $Q$ a unary predicate.     For   $M \models T$, let $M_0$ be the reduct to $L_0$.  We will consider only countable models here.
Let $P$ be the Cohen forcing poset, say with universe $\omega$ in $M_0$.\footnote{We choose Cohen forcing  for definiteness}
Let   $T$ be the expansion of $T_0$ asserting that $Q$ is Cohen generic:   every dense subset of the Cohen forcing notion $P$
(represented by an element of $M_0$) meets $Q$ nontrivially.

Define $M \ee M'$ iff $M_0=M'_0$,    $Q' \in M_0[Q]$ and $Q \in M_0[Q']$.   

Then $\ee$  is a $\bimer$. We use that $Q' \in M_0[Q]$ exactly when: 

\begin{align*}
(M_0,Q,Q')  \models (\exists \tau \in Nset_P)(\forall p' \in P)( p' \in Q' \iff  (\exists p \in Q) p \Vdash \widehat{p'} \in \tau) \end{align*}
Here $\tau$ ranges over the set $Nset_P$ of $P$-names for subsets of $P$, $p,p' $ over elements of $P$, 
and $\widehat{p'}$ denotes the canonical name for $p'$. 
We used the definability in $M_0$ of the forcing relation $\Vdash $.

Clearly $\ee$ has countable classes; given $M_0,Q$, there are only countably many names $\tau$ and so countably many equivalent $Q'$.

  Let us show that it is not ydlept.  Suppose $R \subset M_0^n$ is a definable relation, with the same interpretation in any expansion $(M_0,Q')$ of $M_0$ equivalent to $(M_0,Q)$. We will show that $R$ is $0$-definable in $M_0$.    We can take $n=1$ here.  We 
have $R=\{x: \phi(x,Q) \}$.     If $a \in R$, then $\Vdash \phi(a,Q)$;   otherwise a finite modification $Q'$ of $Q$ will
have $\neg \phi(a,Q')$, so  $a \notin R(M_0[Q'])$, contradicting the assumption on $R$.  Similarly if $a \notin R$ then this is forced by the empty condition.  It 
follows that $R$ is $0$-definable in $M_0$, as $\{x: \ \  \Vdash x \in R' \}$.  
 Compare  \cite[parenthetical example in first paragraph of \S 6]{ah}. 
 
 One can use a similar proof to show that this model equivalence relation is not yclept. Alternatively, one can reference Corollary \ref{cor:ycleptand2ydlept} to note that if it were yclept, it would be ydlept, contradicting the above.
\end{example}

 \begin{example}  \label{ex:fincofinitemodeven} Let $T_0 = Th(B,J)$, where $B$ is the Boolean algebra of finite or cofinite subsets of $\Nn$, and $J$ is the ideal of finite subsets of $\Nn$.
 Note that $\Nn$ corresponds bijectively to the definable set of atoms of $B$, with the 
bijection mapping each number to the corresponding singleton. Let $Q$ be the set of atoms corresponding to even numbers,
 and let $T=Th(B,J,Q)$.    Note that $T$ is bi-interpretable with the product of two disjoint copies of $T_0$, one referring to the finite or cofinite sets of even numbers,
 the other to finite or cofinite sets of odd numbers; these theories admit quantifier elimination in a natural language  \footnote{The theory of
 atomic Boolean algebras with distinguished maximal ideals containing all atoms was determined, perhaps in a slightly different guise,  by Skolem \cite{skolem}
 see \cite{mostowski}}.
 
   Let $\ee $ be the equivalence relation:
 $M \ee M'$ iff $B(M)=B(M'), J(M)=J(M')$, and $Q \sd Q' \in J$.  Again $\ee$ has small classes but is not yclept. Details are left to the reader.

Note that by Proposition \ref{prop:smallclasses}, the example must be $2$-ydlept. The equivalence relation corresponds
to preserving $B$, preserving $J$ and preserving the definable family of sets $Q \sd j$ for $j \in J$.
 \end{example} 

For stable theories, $2$-ydlepts are ydlept.  We may ask:
\begin{question}
If $T$ is a simple theory, is  every model equivalence with small classes  yclept?  
\end{question}

 \begin{rem}  
 Recall that Proposition \ref{prop:small1} states that ydlepts with small classes must arise as a reduct
via constants. We note what happens in the case of a group.
Assume $T$ is an extension of the theory of groups.  In this case it is possible to remove the constant $1$ denoting the identity of the group to obtain a theory $T_0$, recover $T$ as $T_0(1)$
by adding $1$, and be sure that $T \neq T_0$.   This can be extended to expansions of the theory of groups under various assumptions
(originally of stability), and is  often very useful in the theory of definable groups and group actions.

Other examples of reducts by constants arise   among Cameron's five reducts of DLO, and Thomas's reducts of the random graph.
Reducts by constants also play an important role in the study of quasi-finite theories. 

A general criterion 
for when a theory $T$ admits a nontrivial reduct by constants would  be very interesting;
no such criterion is known to us.
\end{rem}

\begin{question}  \label{small-group}   Let $T$ be a complete  theory.  Let $\ee$ be a  $\bimer$ on $\modelsof_{\Om}(T)$
 with small classes.  Does there exist  a $T_\ee$- definable group $G$, a definable set $X$, and an action of $G$ on $P(X)$,
 such that the expansions of $M_0 \models T_\ee$ to $M \models T$ are equi-definable with subsets of $X$, and
  $\ee$ is the orbit equivalence relation?
  \end{question}

In this problem, the action of $\groupoid$ on $P(X)$ should be definable in $\lpluslprime$, but $G$ itself is definable in $L$. A stronger version would ask $\groupoid$ to act
on $X$, with the action on $P(X)$ being the naturally induced one.   A weaker version, allowing a more sophisticated form of action of $\groupoid$ on $P(X)$,
would still be very interesting.

We note that in both examples above, there is a definable group structure:
at least if a sufficiently large fragment of ZF is chosen in the set-theoretic one.
In Example  \ref{ex:fincofinitemodeven},  $\groupoid$ is the maximal ideal of the ``finite'' sets in the Boolean algebra, under the symmetric difference operation.
In Example  \ref{ex:smallnotyclept}, $\groupoid$ is the automorphism group of the complete Boolean algebra associated with the forcing notion,
as interpreted in  $M_0$.   Namely it is known that two generic filters on a forcing notion $P$ that yield the same forcing extension are conjugate,
perhaps not by $Aut(P)$, but by $Aut(B)$.
See Vopenka's theorem 59 in \cite{jech}; or Theorem 3.5.1 in \cite{intermediate}.
If there   exists a fragment of ZF that allows the basic theorems of forcing, but  not  Vopenka's theorem,  it may yield a counterexample to the strong version of our question.

\smallskip

We conclude the section with a  complete description of the dual case, $\mer$s with few  classes.

\begin{lem}\label{fewclasses}  Let $\ee$ be a  $\inftybimer$ on $\modelsof_{\Om}(T)$, with countably many classes.
Then $\ee$ is the reduct to a sublanguage $L_0$ consisting of finitely many finite relations. \end{lem}

\prf   Since a countable set is Borel, $\ee$ admits a Borel section and hence is smooth.  By Theorem \ref{thm:smooth} it is yclept so is determined as the reduct
to a sublanguage $L_0$, possibly of continuous logic.   Let $M \models T$ and let $M_0 =  M \upharpoonright L_0$.   Then for a permutation $\si$ of $\Omega$,
$\si(M) \ee M$ iff $\si(M_0)=M_0$.  Hence there are only countably many models on $\Omega$ isomorphic to $M_0$.
By the Chang-Makkai theorem \cite{changbeth,makkaibeth}, each relation $R$ of $M_0$ is definable with parameters over the pure set $\Omega$.
Now the theory of pure equality admits elimination of imaginaries to the level of finite structures.
Thus the canonical parameter for $R$ can be taken to be a finite set $U_R$, and a finite set $S_R$ of tuples of elements of $U_R$.  It is also easy to see that $\bigcup_R U_R$ must be a finite set $U$: otherwise it could not have countably many conjugates.  Thus $L_0$ is generated by naming an $m$-element set $U$, and possibly a subset of the set
of $m$-tuples of $U$.
\eprf

An alternative proof uses the small index property for the symmetric group.  We see that $\groupoid(M,M)$ has countable index in $Sym(M)$:   
if $s_i \in Sym(M)$ are in distinct cosets of $\groupoid(M,M)$,
then the models $s_i \inv (M)$ are pairwise inequivalent.  Hence by the small index property, for some tuple $\vec m$ of $M$, $\groupoid(M,M)$ contains every element of $Sym(M)$ fixing $\vec m$, and the result follows.

\section{Towards a basis for nontrivial $\mer$s}
\label{sec:basis}

We have given many examples of nontrivial $\mer$s,
but presented no general theory of how such a $\mer$ functions.  Most if not all of the nontrivial $\mer$s we have exhibited include  two nontrivial classes $\CF_1,\CF_2$ of definable sets 
$F,F'$
of $T$, such that equivalence of $M,N$ implies that 
each element of $\CF(M)$ is contained in some element of $\CF_2(N)$.  Nontrivial here means that they are not definable in pure equality.

Here we prove, under mild technical conditions, that a weakening of this must occur in 
every nontrivial $\mer$:  {\em some}  element of $\CF_1$ is contained in {\em some} element of $\CF_2$.  See \corref{cor:mc2}.
In case $T$ has a single sort $V$ and $\acl$ is trivial on $V$, and $T$ eliminates 
$\exists^{\infty}$, we actually get $\CF_1$ to be a family of definable subsets of $V$: see \remref{rem:mc2a}.

\begin{prop}\label{prop:mc2}  Let $T$ be a theory with QE and with elimination of $\exists^{\infty}$.
Let $T+T$ denote the $\lequalslprime$ theory that consists of $T$ on the unprimed copy and a copy of $T$ on the 
primed signature. Then:
\begin{enumerate}
\item     $T+T$ admits a model companion $\widetilde{2T}$.

Moreover, 
we have  $\diamond$:

\medskip

If $(M_1,M_2,f) \models T+T$, $A_i = \acl_T(A_i) \subset M_i$ for $i=1,2$, and $f(A_1)=A_2$ then the quantifier-free diagram
of $A_1 \cup A_2$ in $\lequalslprime$ is complete. Here $A_1$ is in the unprimed signature, $A_2$ in the primed signature.

\item   Let $\ee$ be a $\inftybimer$ on models of $T$. 
Then  $\ee$ is trivial (all models are equivalent)
iff $Th_\forall(\ee) \subseteq  \widetilde{2T}$.
\end{enumerate}
\end{prop}


Towards proving Proposition \ref{prop:mc2}, let $\Phi$
be the set of pairs $(\phi(x_1 \ldots x_n';y) ,\theta(z))$ of variable-partitioned formulas of $L$ such that 
for any  $M' \models T$ and $b' \in \theta(M')$,
 we have  $(*_{b'})$: for any $B' \subset M'$ with $b'$
 from $B'$,  
there exist $a_1',\ldots,a_n'$ in some elementary extension of $M'$ with $a_i' \notin \acl(B')$ and $\phi(a_1',\ldots,a_n',b')$.  Note that $(*_{b'})$ is equivalent to:   there are infinitely many {\em pairwise disjoint} solutions to $\phi(x_1,\ldots,x_n,b')$.
Thus  $\phi, \theta \in \Phi$ means that for a witness $b'$ of $\theta$ we have infinitely many pairwise disjoint solutions for $\phi(x_1 \ldots x_n; b')$.

Consider this condition:
 (*) given $B \subset M \models T$, elements $a_1,\ldots,a_n \in M \setminus \acl(B)$, a formula $\phi(x_1,\ldots,x_n;y)$ and a tuple $b$ from $B$ with
 $ \phi(x_1,\ldots,x_n;b)\in tp(a_1,\ldots,a_n /B)$,  there exists $\theta \in tp(b)$
such that $(\phi,\theta) \in \Phi$.

\medskip

We first argue that:

\begin{claim} (*) follows from elimination of $(\exists^\infty)$.  
\end{claim}

We prove the claim by induction on $n$.  For $n=1$ (*) is just the elimination of $(\exists^\infty)$.
For $n>1$, consider first the case where some $a_i$ is algebraic over another, say  $a_1 \in \acl(a_n,B)$.  Then we may assume
$\phi$ enforces this.  By induction, given $b',B'$, find non-algebraic  $a_1',\ldots,a_{n-1}'$ that witness $(\exists x_n) \phi$,
and choose any $a'_n$  such that $\phi(a_1',\ldots,a_n',b')$.   Then $a_n' \notin \acl(B)$ since otherwise we would have $a_1 \in \acl(B)$, contradicting the assumption.

Thus we assume no $a_i$ is algebraic over another $a_j$ and $B$.  In this case, working over $B(a_1)$, we find $\theta'(y,x_1)  \in tp(a_2,\ldots,a_{n-1}/Ba_1)$
with the property above. 

Let $\phi'(x_2, \ldots ,x_n) =  \phi(a_1,x_2,\ldots,x_n,b)$.
So $\phi'  \in tp(a_2,\ldots ,a_n) / B(a_1)$. 
We use (*) for $\phi'$ to get a formula
$\theta'(y,x_1) \in tp(b,a_1)$ with $(\phi',\theta') \in \Phi$.

We now proceed to prove Proposition \ref{prop:mc2}.
\prf 
We prove the first item.
Let the axioms $A(\phi_1,\theta_1,\phi_2,\theta_2)$,  for  pairs $(\phi_1,\theta_1)$ and $(\phi_2,\theta_2)$ in $\Phi$, assert:

\medskip

 whenever $b_i \in \theta_i(M_i)$, there exist $a_1,\ldots,a_n$ with $\phi_1(a_1,\ldots,a_n,b_1)$ and $\phi_2(fa_1,\ldots,fa_n,b_2)$.

\medskip

The axioms of  $\widetilde{2T}$  include $T+T$ along with the axioms $A(\phi_1,\theta_1,\phi_2,\theta_2)$, for all pairs $(\phi_1,\theta_1)$ and $(\phi_2,\theta_2)$ in $\Phi$.
It is easy to check that any existentially closed model of $T+T$ is a model of $\widetilde{2T}$.
The verification of the strong model completeness asserted in the first item is a routine back-and-forth over algebraically closed $(A_1,A_2,f,f\inv) \leq (M_1,M_2,f,f\inv)$ with $A_i$  algebraically closed in $M_i$.

 We now prove the second part.  Let $(N_1,N_2,g)$ be a     countable model of $T+T$, $\groupoid$-resplendent so that the condition (*) 
 holds for $N_i$ without going to an elementary extension.
    Construct $M$ and $f_i: N_i \to M$  for $i=1,2$ such that  $(M,N_i) \models \widetilde{2T}$. 
    The proof is similar to that of Lemma \ref{lem:ae3}.   
    Now we have $N_1 \ee M \ee N_2$.
    \eprf

\begin{rem} \label{rem:mc2a}  A step of the above back-and-forth goes from $A_1$ to $\acl_T(A_1(a))$ for a single element $a$.  Hence $\widetilde{2T}$ is already axiomatized by
 the axioms $A(\phi_1,\theta_1,\phi_2,\theta_2)$  when 
 one of $\phi_1,\phi_2$ implies that each $a_i$ is algebraic over each $a_j$.
\end{rem}

 \begin{cor}\label{cor:mc2}  Assume $T$ eliminates $(\exists ^\infty)$, and $\acl(0)=0$ in the given sorts. Let $\ee$ be a    $\inftybimer$ on models of $T$, that has more than one class.
Then there  is a definable family $\CF_1$  of infinite definable subsets of the set of distinct $n$-tuples,  and another definable family of co-infinite definable sets $\CF_2$ of
 distinct $n$-tuples, such that $M_1 \ee M_2$  implies that some element of $\CF_1(M_1)$ is contained in some element  of $\CF_2(M_2)$.
\end{cor}
\prf  Since $\acl(0)=0$, $\widetilde{2T}$ is complete.  If it implies $\ee$, then $\ee$ has only one equivalence class, contrary to the assumption.  Otherwise, it is inconsistent with $\ee$, so that $\ee$ implies the negation of some axiom $A(\phi_1,\theta_1,\phi_2,\theta_2)$.   Let
$\CF_1$ be the family of all definable sets $\phi_1(x,b)$ as $b$ ranges over $\theta_1$;
and let $\CF_2$ be the family of all 
definable sets $\neg \phi_2(x,b')$ as $b'$
ranges over $\theta_2$.

\eprf

We discuss possible extensions, variations, and applications of this result.
\begin{itemize}
\item  A similar statement holds for definable partial orders on models.

\item The conclusion of  Corollary \ref{cor:mc2}  does not  look transitive or symmetric.  The most obvious  transitive  condition that implies it
would be: (*)   $\CF_1=\CF_2$, and $M_1 \ee M_2$  implies that {\em every} element of $\CF_1(M_1)$ is contained in some element  of $\CF_2(M_2)$,
and vice versa.

It seems possible that a better approximation to (*) can be obtained by repeating the above idea with two, rather than one, intermediate models between $M$ and $N$.

\item For a general $\mer$, there is no reason to expect $\CF_1=\CF_2$ in (*); witness the 
stable yclept example in an infinite language, with  a family of equivalence relations.    But one can ask:  (**):
Does (*) hold with $\CF_1=\CF_2=\CF$, if $\ee$ is a $\bimer$?  

If so this would give a basis for nontrivial model partial orderings, or model equivalence relations.

 Assuming (**),  we obtain a map 
from $\CF(M)$ to $\CF(N)$ and another in the opposite direction.   Composing, 
we obtain a map from $\CF(M)$ to itself, non-decreasing in the sense of inclusion.  The set of elements that map to themselves forms a family of definable set that is $\ee$-invariant  (a $2$-ydlept situation.)  If this family is empty, every element of $\CF$ is a proper subset of some other element of $\CF$; implying that $T$ has the strict order property.

Thus under (**) and NSOP, the 2-ydlept $\bimer$s form a basis for the nontrivial $\bimer$s.

This  gives some motivation towards Question \ref{question:tamenotpreservenontrivial}. 
\end{itemize}

\section{NSOP and $\forall$-ydleptness}   \label{sec:theories}

Recall that we have made no general claims about which theories are $\forall$-yclept, although we provided an example
in Corollaries \ref{cor:unary} and \ref{cor:forallydleptequivalencerelation}.
 We know of no model-theoretic necessary and sufficient conditions.
We give some partial results here. Recall that a theory has NSOP if there is no formula $\phi(\vec x, \vec x')$ that defines a partial order with infinite chains. 

\begin{prop}\label{prop:yclept-unsop} 

Let $T$ be $\forall$-yclept,   and let $\leq$ be a definable partial  order on the sort $X$.
Then every nonempty definable subset has a maximal element.

 If every $\inftybimer$ on $X$ is yclept, then no definable partial order on $X$ has an infinite strictly increasing chain. Hence if   every $\inftybimer$ on any finite product of sorts of $T$ is yclept, then $T$ has NSOP.  
\end{prop}

\prf  
We start with the first part, assuming $T$ is $\forall$-yclept.
  We consider a definable partial ordering $\leq $  of a nonempty definable subset of the
sort $X$. For simplicity we assume the order is on all of $X$.
  For $a \in \Om$, let  $a^+ :=\{b \in \Om: a \leq b\}$.  Let $X_{max}$ be the set of maximal elements of $X$, i.e. 
  \[
  \{x \in X:  (\forall y \in X)(x \leq y \to x\geq y)\}
  \]

 \begin{claim}   $X_{max} \neq \emptyset$.
\end{claim}
 \prf  Suppose $X_{max}=\emptyset$.  Consider two models $M,N$ of $T$ with the same universe $\Om$.  Let
 \begin{align*}
     C(M,N) = \\
     \{a \in \Om:  (\forall x,y)(a \leq_M  x,y \,  \bigvee \   a \leq_N x,y)   \to  y \leq_M  x \iff y \leq_N x \}
      \end{align*}
So $C(M,N)$ is closed upwards for both both $\leq_M$ and $\leq_N$, and they agree on $C(M,N)$.
Define $M \ee N$ iff $C(M,N)$ is cofinal in both $M$ and $N$.  We can easily verify that this is an equivalence relation. We use that  $C(M,N) \meet C(N,N') \subset C(M,N')$.
If $a \in M$, then there exists $a' \in C(M,N)$ with $a \leq_M a'$, and $a'' \in C(N,N')$ with $a' \leq_N a''$.  So $a'' \in C(M,N)$ too, and with $a' \leq_M a''$, etc.
We will show that $\groupoid(M,M)$ is dense in $Sym(\Om)$.  It suffices to show that any permutation $\si_F$ of a finite set $F \subset \Om$
extends to an element of $\groupoid(M,M)$.   For each $a \in \Om$, choose $\beta(a) \geq a$ such that no element of $F$ is $\geq \beta(a)$. This is easy to
do using the condition $X_{max}=\emptyset$.  Let $C^* = \union_{a \in \Om} \beta(a)^+$.  Then $C^* \meet F = \emptyset$.  Extend
$\si_F$ to a permutation $g$ of $\Om$ fixing $C^*$.  Let $N=^g M$, i.e. $g: M \to N$ is an isomorphism.  Then $C(M,N) \supseteq C^*$, and 
$M \ee N$, i.e. $g \in \groupoid(M,N)$.  Thus  we have proved density. Since $T$ is $\forall$-yclept, we know in particular that this $\mer$ is yclept.  By Theorem
\ref{thm:aerelative} $\groupoid(M,M)$, being closed and dense, contains every permutation. 
Thus the equivalence relation has either one class or just singleton classes.
In both of these cases, any definable set has a maximal element, namely any element, contradicting the assumption.   
\eprf

This proves the first item.  For later use, note that since $T$ is $\forall$-yclept, the claim above is valid not only for $(X,\leq)$ but also for the induced partial ordering on $X',X'',\ldots$ as long as they are nonempty.

 The proof of the second item is very similar.  We prove the contrapositive.  Thus assume $\leq$ is a definable partial ordering of
$X$ with an infinite chain.  It follows that for some complete type $P$, $\leq $ has no maximal element on $P$.
$P$ is the intersection of definable sets $D_k$.
Define $C(M,N)$ by exactly the same formula as in the proof above.  But now define $M \ee N$  if $D_i^M=D_i^N$ for each $i$, and $C(M,N)$ is cofinal in both $D_i(M)$ and $D_i(N)$.
So $\ee$ is $\bigwedge$-definable.  If $\ee$ holds then
$P^M=P^N$,  and  $P \meet C(M,N)$ is cofinal
in $D_k^M$ and in $D_k^N$ for every $k$. Hence, if $M,N$ are slightly saturated, $P \meet C(M,N)$ is cofinal
in $P(M)$ and in $P(N)$.  The proof that $\ee$ is a
 $\mer$ is the same as in the previous item.   Let $M$ be a countable resplendent model of $T$.  $\groupoid=\groupoid(M,M)$ acts on $P$, and the same proof as in 
the first item  shows that $\groupoid$ induces a dense  subgroup of $Sym(P)$.
  If $\ee$ were yclept,  then $\groupoid$ would be closed and so would induce $Sym(P)$.  But clearly it is possible to construct a permutation
 of $P$ that does not extend to any element of $\groupoid$, since it is not order-preserving on any final segment of $P$.
 Thus  $\ee$ is not yclept.
    \eprf

 \begin{rem}  If  there is a unique $1$-type on $X$, then the $\mer$ constructed in the proof above is actually definable, not just $\wedge$-definable. If $T$ is $\aleph_0$-categorical and has a definable partial order on $X$ with infinite chains, we have exhibited a non-yclept  $\bimer$ on it. \end{rem}

We remark on a higher-dimensional generalization, with the same proof, mutatis mutandis.   Assume $\leq$ is given on $X^k$.   By a {\em hyperplane} of $X^n$
we just mean a subset defined by $x_i =a$, where $x_i$ is one of the coordinates.  We   call a definable set $Y$ meager if for any $a$, for some $b \geq a$,
$b^+ \meet Y = \emptyset$.
Then Proposition \ref{prop:yclept-unsop} generalizes to:  Let $\leq$ be a definable partial ordering on $X^n$, such that any hyperplane is meager.
Then $T$ admits a non-yclept $\mer$.

We believe that the proof of Proposition \ref{prop:yclept-unsop} can be extended
to show that if every 
$\inftybimer$  of $T$ is $k$-ydlept for some $k$, then $T$ has NSOP.

\section{Examples} \label{sec:examples}

In this section we give a number of examples.
In each case, we discuss the equivalence relation and the corresponding groupoid.
The examples point to the different ways to approach $\mer$s: for example, through
giving the equivalence relation, through giving the groupoid, or for showing a complete set of invariants.

\ssec{First-order reducts}   Let $L_0$ be a  reduct of $L$, and $\ee_0$ the equivalence relation that preserves $L_0$. 
Clearly, the invariants are   the  $L_0$ reduct.  If $L_0$ is finite, this is a $\bimer$.
The corresponding groupoid  is the $T$-groupoid whose morphisms are $L_0$-isomorphisms.

\ssec{Definable topologies}\label{ex:deftop} 

Let $\tau$ be definable family of definable sets, forming a basis for a topology in $T$  \cite{flumziegler}.
The invariants are   the set of open sets, and
the groupoid is the set of   homeomorphisms.
For $DLO$, with the usual definable topology,  the maximal CL reduct is trivial. 

For the specific model $\Rr$, any automorphism preserving the topology preserves also the ``betweenness'' relation  of triples.
But for $\Qq$, the homeomorphism group is dense.  For instance interchanging  $(-2 \pi,- \pi)$  with $(\pi,2\pi)$ while fixing  the rest is a homeomorphism.
Thus it is not the case that any automorphism of any model fixing $\tau$ also fixes the betweenness relation.

It follows easily that the maximal CL reduct is trivial.   Further, no definable family of definable sets is invariant, unless it is definable from pure equality.
Hence the $\mer$ is not $n$-ydlept for any $n$.
The fact that the $\mer$ is not $n$-ydelpt also follows from the facts that:
\begin{inparaenum} \item the CL reduct is trivial, hence in particular the $\mer$ is not ydlept.
\item plus Proposition \ref{prop:stable-n} and Proposition \ref{prop:nipctistem}, which tell us that DLO is istem, so a $\mer$ that is $n$-ydlept would need to be ydlept. 
\end{inparaenum}

Note that Proposition \ref{prop:maxred} implies $n$-transitivity of homeomorphisms on saturated models, for each $n$;
but not for an arbitrary model, witness $\Rr$, where orientation is preserved.

\ssec{Uniform continuity}   \label{ex:unifcont} Let $T=Th(\Qq,+,<)$.   Consider the category of uniformly continuous maps:
\[ (\forall u>0)(\exists v>0) (\forall x)(\forall x') (|x-x'|<v \to (f(x)-f(x')|< u ) \]
For maps from $M$ to $N$, $v,x,x'$ range over $M$ but $u \in N$.

Let $\groupoid$ be the groupoid of  invertible morphisms in this category.     The same can be done on a definable set $X$;
take $X$ to be the  plane (the set of pairs).  Then even looking at just the single model $\Rr$,
it is clear that  the maximal CL reduct is trivial.

It is clear that no definable family of definable subsets of the plane $\Pi=\Rr^2$ can    be preserved by all uniformly continuous maps, unless it is definable in pure equality on $\Pi$.  In particular a family that is preserved by all such maps must consist of finite and co-finite sets.
For if $X\subset \Pi$ is neither finite nor co-finite, it will have a 1-dimensional boundary which is a finite Boolean combination  of lines and points, but with at least one line or interval among them.  Then a uniformly continuous
map can be found that moves the interval to an arc of circle say, or in any case something nonlinear.   
This shows that the  equivalence   relation induced by uniformly continuous maps on $\Pi$ is not  $n$-``unary ydlept'' for any $n$, i.e. where only
definable families of subsets of $\Pi$ itself are allowed.

 It should not be difficult (though possibly a little tedious) to extend this and 
 show that no definable family of definable subsets of  $\Pi^n$ is preserved, unless it is definable in pure equality alone.    Given this, it follows  that the  equivalence  
relation induced by uniformly continuous maps on $\Pi$ is not  $n$-ydlept for any $n$.

 \ssec{Higher-order internal set theory}  \label{ex:ho} 
 
 We give an example of an $n$-ydlept equivalence relation over a sort $P$, 
that  is not $(n-1)$-ydlept over sort $P$.    
$T_n$ has sorts $P,P_1,\ldots,P_n$ with certain graph structures on them;  the equivalence relation $\ee_n$ looks only at $P$ and a certain  
 $n$-subset of $P$, associated with the graph structures.  In particular, two models with the same $P$ and same $n$-set on $P$ will be $\ee_n$-equivalent.

We construct $T=T_n$.   The language will have  a ``main'''
sort $P=P_0$ and   additional sorts $P_i$ for $0<i<1+n$. There will be  binary relations 
$\G_i \subset P_i \times P_{i+1}$ for each $0 \leq i<n$.    The axioms state that:
\begin{itemize}
    \item 
$(P_0,P_1,\G_0)$ and $(P_{n-1},P_n,\G_{n-1})$  are random bipartite graphs;
\item
  for $0<i<n$:  for any  disjoint  finite $A, A' \subset P_{i-1}$ and $B,B' \subset P_{i+1}$, there exists an element $c \in P_i$
such that for all $a \in A, a' \in A', b \in B, b' \in B'$ we have  $(a,c) \in \G_i, (c,b) \in \G_{i+1},  (a',c) \notin \G_i, (c,b') \notin \G_{i+1}$.  
\end{itemize}
It is easy to check that $T_n$ is complete, $\aleph_0$-categorical with QE.

A $1$-block on $P$ is a subset of the form $\G(b)=\{a \in P: (a,b) \in \G\}$, where
$b \in P_1$.   Similarly we define $2$-blocks as sets of blocs represented by elements of $P_2$, etc.   
Let $\groupoid_n$ be the  $\bigp$ given by the  groupoid  of bijections $P \to P$ respecting the families
of $i$-blocks  for each $i \leq n$.    Thus the associated $\bimer$ $\ee_n$ makes $M,N$ equivalent if the $m$-sets described above are equal for $M$ and for $N$, for all $m \leq n$.  

By definition, $\ee_n$ is $(n+1)$-ydlept, indeed an $(n+1)$-set reduct. As with any $(n+1)$ set over sort $P_0$, we can consider it
as a $\mer$ with only one coupled sort, $P_0$, and the other sorts decoupled.
 Thus, in considering the class of models that can be considered equivalent, we do not restrict the universes of the
sorts $P_1 \ldots P_n$.

We conjecture that the $\ee_n$ cannot be expressed as an $n$-set  $\mer$.
Or equivalently, by  Lemma \ref{lem:special},  that it separates
$(n+1)$-ydlept $\mer$s from $n$-ydlept $\mer$s. Below we present only a partial result, showing that
it separates $(n+1)$-ydlepts over the given $P_0$ from $n$-ydlepts with the same $P_0$.
Equivalently, the argument below will give an $(n+1)$-ydlept $\mer$ over coupled sort $P_0$ that is not an $n$-ydlept over coupled sort $P_0$.

    \begin{lem} \label{ydleptLS}  Let $T$ be a countable theory, $\ee$ an $(n+1)$-set $\mer$ on $T$ over coupled sort $S_0$.  
Let $M \models T$ with $|S_0|$ countable.
Then there exists $M' \prec M$, $|M'| \leq \beth_n$, such that $M \ee M'$.
\end{lem} \prf  The equivalence relation $\ee$ requires having the same $n$-set for a certain $(n+1)$-ary relation $R$.
By L\"owenheim-Skolem, there exists $M' \prec M$ with the same $n$-set as $M$, and $|M'| \leq \beth_n$.   \eprf

\begin{cor} \label{randomline-sorts} The $\mer$ $\ee_n$ 
defined cannot be induced by an $n$-set over coupled sort $P_0$.
\end{cor}
\prf Let $M \models T_n$ with $|P_i(M)| = \beth_i$.   If  $\ee_n$ were $n$-ydlept, by \lemref{ydleptLS} there would be $M' \prec M$,
$|M'| \leq \beth_{n-1}$ with $M \ee M'$.  But this implies that $M,M'$ have the same family of $n$-blocks, so $|M'| \geq \beth_n$, a contradiction.
\eprf

\begin{rem} \label{rem:randomline}
We give some additional observations on this family of examples:
\begin{enumerate}
\item   $T_1$ is the theory of random bipartite graphs.   Any $T_n$ is interpretable in $T_2$.   Namely if $M=(P_0,P_1,P_2; \G_0,\G_1) \models T_2$,
define $P_3:=P_1, P_4:=P_0,  P_5=P_1\ldots $. So $P_{i}=P_1$ for $i$ odd, $P_{4j} = P_0$, and $P_{4j+2}=P_2$.   This is a model of 
$T_\infty= \bigcup T_n$ since the axioms concern only consecutive triples of $P_i$'s, and are verified in $T_2$.

\item  Similarly, the random graph with two edge colors --- i.e. the model completion of the theory
of two disjoint binary relations $(P,E_1,E_2)$ --- interprets each $T_n$. Let $P_k=P$ for each $k$,
and let the edge relation between $P_k$ and $P_{k+1}$ be $E_1$ if $k$ is even, $E_2$ if $k$ is odd.

The random graph with two constants interprets the random four-partite graph, and thus interprets each $T_n$.

Can we combine the higher-ydlept $\mer$s $\ee_n$ on each $T_n$ exhibited above to find a $\bimer$ on the random graph that is not $n$-ydlept for any $n$?

\item    Simon Thomas'  conjecture \cite{simonthomasconjecture}, that there are only finitely many reducts of certain theories $T$, would imply that there are only finitely many ydlept $\mer$s 
 on such theories. Thomas proved it for the random graph, and Lu \cite{yunlu}
extended  this to the bipartite random graph $T_1$.   We are not sure if reducts of $T_2$ have been looked at; in any case  the  example shows   that $T_2$ has an infinite chain of non-ydlept $\mer$s.

\end{enumerate}\end{rem}

 \ssec{A quasi-finite topology}  \label{ex:qft}  We give an example  similar to  
the previous Example \ref{ex:ho} for $n=2$.  But, like the example in Subsection \ref{ex:deftop}, it can be formulated topologically.  
  Let $T$ be the theory of a vector space over $\Ff_p$
with nondegenerate symplectic form.   

Let $V$ be a saturated model. For example,  the unique countable model.    We have a topology on $V$ generated
by the subgroups of finite index $V_a$, where $V_a =\{v \in V: (a,v)=0 \}$.  Any open subgroup of index $p$ can be shown 
to have the form $V_a$ for some $a$.  Let $\groupoid$ be the group of homeomorphic vector space automorphisms. This contains
$Aut(V)$ and is definable, since any $V_a$ must go to some $V_{b}$.  
By definition, this  is the 
 $2$-ydlept equivalence relation generated by the definable family $V_a:a \in M$, along with the definable set $+$.

The first order structure preserved is just $+$ which
does not explain $\groupoid$.  So it is not yclept.
Note that this example is quasi-finite, and thus quasi-finite theories can have $\bimer$s that are not yclept.
   Here the topology does not have a definable basis, but the homeomorphism groupoid turns out to be definable anyway.

\ssec{Definable differential geometry}
\label{ex:diffgeo}

 Let $T$ be the theory with $2$ sorts. One sort $R$ supports a real closed field structure, and the other   sort $X$ being  some fixed 
definable subset of $n$-space over a real closed field, with the induced structure: for example, the $2$-sphere as a subset of $3$-space.  
 
Consider first the case that $X$ is just the segment $(0,1)$ in $1$-space.
Let $\groupoid(\Psi)$ be the groupoid of bijections $f$ that are isomorphisms on the $R$ sort, and -- identifying the two fields via the isomorphism --  acting by $k$-times differentiable diffeomorphisms on the $X$-sort.  
Consider, for example, $n=k=1$,

\begin{align*} (\forall x \in X)(\exists c \in R) (\forall \e \in R')(\exists  \delta \in R) (\forall y \in X)\\
(|x-y|<\delta) \implies (|f(y)- f(x)- c \cdot (y-x) | < \e \cdot |y-x| )  
\end{align*}

There are many rich variations here, e.g. asking to preserve Riemannian structure.
This may give a new way of thinking about them model-theoretically.

Many of these geometric structures can be presented either via (charts and local) isomorphisms of a specified type,
or via sheaves, which can be thought of as higher-order reducts.  
 For 
the differentiable structure above, invariants can be taken as the differentiable functions on $X$ into $R$.
See \cite{cartanconnection,cartangeometry}.

\ssec{Cofinal higher order set theory}  \label{ex:cof} Here we assume, as above, that  $T$ has $n$ disjoint predicates $P=P_0,\ldots,P_n$ and random graph structures $\G_i$ on $P_i \times P_{i+1}$ for each $0\leq i<n$.   But now $P$ is also assumed to carry a partial ordering $\leq$.   
We define partial orderings $\leq_k$ on $k$-sets in $P$ as follows:  $0$-sets are elements of $P$ and we let $\leq_0 = \leq$.   If $a,b$ are $(k+1)$-sets, thus a set of $k$-sets,
we define $a \leq_{k+1} b$ iff $(\forall u \in a)(\exists v \in b)(u \leq_k v)$.  

We can define a partial order relation on models of $T$ with the same universe:  $M \leq M'$ iff the $(n+1)$-set consisting of
$n$-sets coded in $P_n$  in $M$ is $\leq_{n+1}$ than the corresponding one of $M'$.
This give rise to a $\mer$ defined by $M \ee M'$ exactly when $M \leq M' $ and $M' \leq M$.

Already for the case $n=0$ this  includes the example in Subsection \ref{ex:deftop}:  Start with the partial ordering on pairs $(a,u)$
where $a$ is a point and $u$ an element of the definable basis, with $a \in u$; where $(a,u) \leq (b,v)$ iff
$a=b$ and $u \subset v$.

It would be interesting to see a natural set of invariants, for example, based on a notion of second order reduct,
generalizing the way that a {\em topology} is the invariant of cofinal equivalence on bases.

This example cannot be captured by an $n$-ydlept higher set theory for any $n$, as in Example \ref{ex:ho},  since 
 e.g. for the topology of the example in Subsection \ref{ex:deftop} on $\Qq$, no definable family of definable sets  can be invariant.
It is not clear if analogues can exist for simple or even stable theories.

\ssec{Continuous logic fragments} \label{ex:CLexample}
We give an example of a (stable) theory $T$ and $\bigwedge$-definable equivalence relation on models of $T$, 
 given by equality of certain continuous logic reducts, 
 whose maximal \emph{first-order} invariant reduct is trivial.

Let $T$ be a theory with infinitely many independent unary predicates $P_n$, $n=0,1,\ldots$.
Let $a=(a_n)$ be a sequence of positive reals whose partial sums converges.   Let $\phi_n$ be the characteristic function of $P_n$,
and define $\psi_a(x) = \sum_n a_{n} \phi_n$.    Let $\groupoid_a$ be the groupoid preserving $\psi_a$.   It can be checked
to be $\bigwedge$-definable.  E.g. for the $b_n$ below, equivalence of a primed and unprimed model  holds iff for each $n$, for all $x$,
$|\sum_{i <n } \phi_n(x) - \sum_{i<n} \phi'_n| \leq 2^{-n}$; this last expression is just a certain Boolean combination of $P_i(x)$
and $P_i'(x)$ for $i <n$. Thus its maximal CL-approximation, as given
by Proposition \ref{prop:maxred}, is itself.

In case $a_n=3^{-n}$, we have $\groupoid_a = Iso_T$, i.e. the full structure can be recovered.
But if $b_n=2^{-n}$, $\groupoid_b \neq Iso_T$. Any permutation mixing up elements lying in   $P_0 \setminus (\union_{k\geq 1}P_k)$
 with elements lying in $\meet_{k \geq 1} P_k \setminus P_0$ is an automorphism.  
It is clear that $\groupoid_b$ does not preserve any   relations that are first-order definable 
 in $T$.

 \ssec{An yclept but not ydlept example over a stable theory} \label{ex:stablex}  
Recall that in Section \ref{sec:notydlept} we have given a stable theory with a $\bimer$ that is yclept but not ydlept. Here
we provide a much simpler example of something that is yclept but not ydlept in a stable theory -- but one that is not a $\bimer$.

\begin{example} \label{ex:easyycleptnotydlept}
Let $T$ be the (stable) theory of equivalence relations $E_q$ for $q \in \Qq$, with $E_p \subset \ee_q$ if $p<q$.
Let $\ee$ be the equivalence relation on models of $T$ making two models $(\Om,E_p)_{p \in \Qq}$ and $(\Om,E'_p)_{p \in \Qq}$ equivalent if for any $p<q<r \in \Qq$, $E_{p} \subset E'_q \subset E_r$.  This is $\bigwedge$-definable.  

The maximal CL reduct of any finite sublanguage is just the equality.
However, the example is still yclept.   For any $a,b$
the CL binary relation with value $\inf \{ p \in \Qq : a E_p b \} \in \{-\infty\} \union \Rr \union \{\infty\}$ is definable,
and determines $\ee$.

A similar example with unary predicates is also possible.  
\end{example}

   \ssec{The Lipschitz groupoid} \label{ssec:lipschitz} 

We give an
 example of an equivalence relation which is definable by a  $\Sigma_2$ sentence, which we refer to as the
    \emph{Lipschitz groupoid}.    We have the theory of real closed fields, presented via the field sort $R$ as well as another sort $X$
   which is just another copy of $R$.    $T$ knows about the identification of $R$ with $X$, and in particular includes the function
   $(x,y) \mapsto |x-y|$, from $X^2$ to $R$.    

We consider the category 
   whose objects are models of $T$, and whose morphisms are   bijections $g: M \to M' \models T$ that are field isomorphisms on $R$,     and Lipshitz on $X$:  
   \[
   (\exists \delta \in R^{M'} ~ \delta>0)(\forall x,x' \in X^M)(  |gx-gx'| \leq \delta \cdot g(|x-x'|) )
   \]
   The Lipschitz groupoid   consists of the invertible maps in this category.  We leave to the reader the verification that this is indeed a category: that is,  closed under composition.

It is easy to see that the maximal yclept reduct has at most the ordering relation on  $X$, by density: for any two increasing $n$-tuples in $X$ there is a Lipschitz map taking one to the other, e.g. a continuous piecewise linear   map, with finitely many slopes.     But there are many non-Lipschitz maps that preserve the ordering: for example, a map with slope $n$ on $[n,n+1]$.      Hence this relation is not yclept, and thus is not $\Pi_2$-definable by Theorem \ref{thm:ae}.

Note that two sorts $X,R$ are needed since while they look the same to $T$, the identifying map and even the distance map $X^2 \to R$ is not preserved
by the groupoid.

\subsection{Quantifier complexity} \label{Quantifier complexity}

Let us summarize the examples from    this section,
keeping in mind
that the non-yclept ones cannot be $\Pi_2$ by Theorem \ref{thm:ae}.    
Some basic observations are:
 
\begin{itemize}
\item Every yclept  $\bimer$
is universal; and we saw that conversely universal $\mer$s are yclept.  
  
 \item The  basis-definable topologies example, as in  \ref{ex:deftop} is $\Pi_3$.  Witness the definition of continuity in   metric space style:
   \[ (\forall x)(\forall \e)(\exists \delta)(\forall y)(d(x,y)< \delta \to \ldots)\]
   Likewise uniform continuity (Example \ref{ex:unifcont}), as well
   as Example \ref{ex:qft}.  

\item As noted in Subsection \ref{ssec:lipschitz}, the Lipschitz example is $\Sigma_2$.
   
  \item The differential examples  \ref{ex:diffgeo} are, as written, $\Pi_5$.   Higher order differentiability, if written naively, leads to formulas of progressively higher complexity. 
  
  On the other hand if we define the $\mer$ using  polynomial approximations - asserting existence of a polynomial of degree $d$ approximating the function to $O(d+1)$ - 
  one stays at a bounded complexity level for any order of differentiability. 
  
  The latter $\mer$ is a refinement of the former; to show that they are equivalent would require a proof of Taylor's theorem that remains valid in nonstandard models; the most standard proof
  does not have this property, since it invokes Cauchy's mean value theorem.  The ``alternative proof'' in the wikipedia entry on Taylor's theorem, using l'Hopital's rule, does fit the bill, at least in the one-variable case.

  \item The $n$-ydlept examples   \ref{ex:ho}.n. are typically $\Pi_{2n-1}$. 
   
   \item 
     The cofinal set theory examples in Example \ref{ex:cof}.n. have quantifier complexity at most $\Pi_{2n+2}$
       \end{itemize}

 \ssec{Separating ydlept and higher ydlept in the finite}
 
 We comment briefly on  the case where we restrict to \emph{finite structures}, with empty theory.
In this context, we give a direct proof that there is a $2$-ydlept equivalence relation that is not  ydlept. We believe
that the proof idea extends to  separate $(n+1)$-ydlept from $n$-ydlept in the finite, but have not verified this.

Consider the theory with two sorts $P$ and $Q$, $G(x,y)$ be a binary relation with $x$
of sort $P$ and $y$ of sort $Q$,  and $T$ the theory of the random bipartite graph. Given
$x_0$ of sort $P$ in some model of $T$, the \emph{adjacency set of $x_0$} is the set of $y$ such
that $G(x_0,y)$ holds.  We can then talk about the  adjacency sets of a model $G$: the set of
adjacency sets of $x$ as $x$
ranges over sort $P$ interpreted in $G$. 
We consider the $2$-ydlept $\ee$ generated by the partitioned formula $G(x;y)$. That is, two models $G$ and $G'$
are equivalent  if they have the same adjacency sets.
\begin{prop} \label{prop:finitesep} $\ee$ is not ydlept over finite models.
\end{prop}

Given first-order open formula $\psi$, let $C_\psi$ be the class of graphs that satisfy sufficiently many axioms for the
random graph. We will need the number of axioms to be high enough that $\psi$ is equivalent to
a quantifier-free formula $\phi$ for all graphs in $C_\phi$. 

We write $\phi$ as $\bigvee_i A_i$ where $A_i=\bigwedge_j \phi_{ij}$ where $\phi_{ij}$
is an equality atom, a relational atom, or the negation of a relational or equality axiom.
We further decompose the $A_i$ into a subset $A$ that include a relational atom
and a subset $B$ that do not include a relational atom.

We also can assume:
\begin{itemize}
\item the $A_i$ are pairwise contradictory
\item each $A_i$ specifies the equalities among $\vec x$
\end{itemize}

The first case to consider is where  $A$ is empty.

\begin{lem} If $A$ is empty, there
are $G, G'$ that agree on the $\phi$ reduct and disagree on some adjacency set.
\end{lem}

\prf Take any $G$ and $G'$ with the same domain but differing on some adjacency set.
\eprf

Note that if $A$ is non-empty, there will be witnesses to disjuncts in $A$ in sufficiently large models.
And since disjuncts are disjoint, no disjuncts in $E$ will be satisfied by these tuples.

\begin{lem} For each $G$ in $C_\phi$, for each $\vec x_0$ in  $G$ satisfying a disjunct
of $A$ in $\phi$, there
is $G'$ containing $\vec x_0$ with the same adjacency sets as $G'$,
 with $\vec x_0$ not satisfying $\phi$ in $G'$.
\end{lem}
\prf
For each  $A_i$ in $A$ satisfied by $\vec x_0$ in $G$ we choose $B_i$ a relational atom or the negation
of such in $A_i$ such that $\vec x_0$ satisfies $B_i$, and let $B'_i$ be the corresponding ground atom with
$\vec x_0$ substituted.
We will construct $G'$ so that $\vec x_0$ does not satisfy $B_i$ in $G'$.
Since the disjuncts of $E$ are not satisfied in $G$, they will not be satisfied in $G'$.

We let $c_1 \ldots c_k$ be the elements of $P$ occurring as the first elements  of such an atom.
Let $S_i$ be the adjacency set of $c_i$.
For each $c_i$, we let $D_i$ be the set of elements $d$ such that $G(c_i, d)$ occurs
as an atom in $B_i$ and $E_i$ be the elements $e$ such that $\neg G(c_i, e)$ occurs as an atom of
$B_i$. Thus $S_i$ contains $D_i$ and is disjoint from $E_i$.

We claim that there is $c'_i$ whose adjacency set contains each $E_i$ and is disjoint from $D_i$.
This follows easily from the fact that $G$ is sufficiently random.
Let $G'$ be formed from $G$ by swapping the adjacency sets of each $c_i$ and each $c'_i$.
By construction, the adjacency sets represented in $G'$ are the same as those in $G$.
\eprf

Putting the two lemmas together proves Proposition \ref{prop:finitesep}.

\section{Some additional questions}  \label{sec:questions}

While we have examined examples of definable equivalence relations that are not ydlept, and not even yclept or $k$-ydlept. But we know little about
which theories admit such examples. 
Even with very strong hypotheses on the theories we cannot say anything about all $\bimer$s.
\begin{question}  
 \label{question:tameforall}  Are  totally categorical theories $\forall$-ydlept?  Or $\forall$-yclept? Or even ``$\forall$-$\omega$-ydlept: every $\bimer$ is an $n$-ydlept for some $n$.
\end{question}
We do not know the answer for a single non-disintegrated strictly minimal set. Nor do we know it for the Morley rank two theory, interpretable over pure equality,  of pairs of elements from an infinite set.
In the opposite direction, we do not know a counterexample even for superstable $T$.   For stable theories, we know from Proposition \ref{prop:stable-n} that higher ydlepts are ydlept. And we know it is possible to have yclept equivalence relations that are not ydlept:
see Example \ref{ex:stablex}, for $\inftybimer$s,  and the example from Section \ref{sec:notydlept}.  But   at this level too we do not know if non-yclept definable equivalence relations are possible.

We do not have example with $\bimer$s that are not $n$-ydlept for any $n$ even weakening the hypotheses on the theories considerably.

For example, moving out to NSOP theories, for all we know every $\bimer$ could be $n$-ydlept for some $n$.

\begin{question} \label{question:tamenotpreservenontrivial}   Let $\ee$ be a nontrivial $\bimer$ on the models of an NSOP, $\aleph_0$-categorical  theory $T$.  
Must   $\ee$ preserve a definable family of definable sets,
not definable in pure equality?    

\end{question}

We do not know the answer for the random graph.  Without $\aleph_0$-categoricity,  \ref{ex:stablex} is a counterexample.

 We also know almost nothing about what happens over finite structures.  In particular, our examples of $\bimer$s that are not $n$-ydlept for any $n$ are all on
 infinite structures.

 We can show that for every $\bimer$ there is an ydlept $\mer$ that agrees with it on finite structures ``modulo some small distortion''.
 
Let  $Mod_{fin}(T)$ denote the category of finite models of $T$, with morphisms being isomorphisms.

\begin{prop}  \label{prop:interpfinite} Let $L$ be a finite language, let $T$ be an $L$-theory and let $\ee$ be a $\bimer$.  Then there is an ydlept $\mer$
$(T_1,\ee_1)$  and an interpretation $\psi$ taking models of  $T_1$ to models of $T$ such that for every pair of models $M_1, M'_1$ of $T_1$, 
$\psi(M_1) \ee \psi(M'_1)$ iff $M_1 \ee_1 M'_1$, and
$\psi$  induces an isomorphism between  $Mod_{fin}(T_1)$ and $Mod_{fin}(T)$.  
\end{prop}  

\prf   We know that $T$ is bi-interpretable with a theory of graphs; so we may assume $L$ consists of a single sort and a single binary relation $R$.
$L_1$ will have a new predicate $A$, the relation symbol $R$ on $A^2$, and defining $B$ to be the complement of $A$, also a ternary relation $\e \subset A^2 \times B$.
The theory $T_1$ says   $\epsilon$ is a ``membership'' relation on $A^2  \times B$, in the sense that extensionality holds,    any singleton set exists, and the union of two sets exists. 
Let $ext(b)$ denote the set $\{a \in A^2: a ~ \e ~ b\}$. Then an example of extensionality would be  $(\forall a \in A^2)(\exists b \in B) (ext(b)=\{a\})$, etc. If $N \models T$, we write 
  $A(N)$ for the set defined by $A$, but also for  the $L$-structure $(A^N,R^N)$. 
  
We now give the definition of $\ee_1$. Let 
\[ D=  \{b \in B: (A, ext(b)) ~ \ee ~ (A, R) \}
\]
This is a definable subset of $B$ in $N$.   Given  two models $N,N'$ of $T'$ and  a bijection $g: N \to N_1 $ the pair $(N, g N')$ is in the groupoid $\groupoid$ corresponding to $\ee_1$ 
   iff  $g$ preserves the $0$-definable sets $A, B,\epsilon,  D$.    

 The interpretation $\psi$ is just the forgetful functor, forgetting about $B$ and $\e$.   
 On finite models, it is clear that this functor is
 an equivalence, i.e. any finite $M \models T$ is obtained in this way, and if $N_1,N_2 \in Mod_{fin}(T_1)$ then any isomorphism between $\psi N_1 $ and $\psi N_2$  lifts to an isomorphism $N _1 \to N_2$.

For finite $N_1,N_2 \models T_1$, a bijection $g: N_1 \to N_2$ lies in the groupoid corresponding to $\ee_1$ iff it preserves $A$,$B$,$\epsilon$ and the restriction to 
$A(N_1) \to A(N_2)$ is in the groupoid  of $\ee$.

 \eprf
 
Note that the size increase in applying $\psi \inv$ is simply exponential in a polynomial of $n$. More precisely, for some $k$,  for a finite model of $T$ with $n$ elements, $\psi \inv(M)$ has size $\leq 2^{n^k}$;  even in case the $\mer$ is $4$-ydlept, say.

\begin{question} \label{question:fin}  
Let $\ee$ be a  $\bimer$ over the empty theory. Is $\ee$ equivalent over finite structures to an $n$-ydlept model equivalence relation?

Note that the hypothesis can be weakened to: 

\medskip

$\ee$ is given by an $\lequalslprime$ sentence that defines
an equivalence relation over finite models, or over all models.

\medskip

In either case, we do not know the answer.

\end{question}

And the same question for a quite different, highly structured world:  

\begin{question} \label{question:psf}
Let $M$ be   smoothly approximable.  Let $C$ be the class of finite homogeneous substructures of $M$.  Let $\ee$ be a definable model equivalence relation on $C$.  Is $\ee$ equivalent on $C$ to an $n$-ydlept one? 
\end{question} 

 Compare Example \ref{ex:qft}; the topological equivalence relation there is not ydlept, but is equivalent to an ydlept equivalence relation over finite structures, namely, the one preserving $+$.
 
\begin{rem} \label{question:finitelygen}   Let $\ee$ be  $\bimer$  whose maximal CL reduct $R$  is essentially  DL.   
 Must  the $\mer$ generated by $R$  be definable?  

We  outline a construction that, we believe will give a negative answer.
 Construct a bipartite graph $(D,Q,I)$ such that infinitely many elements of $Q$ are $0$-definable.  One way to do this would be to ensure
  there is a unique $c_n \in Q$ with $|D(c_n)|=n$.   Let $\ee$ be the $2$-ydlept $\bimer$ that makes $M,N$ equivalent iff 
  $\{D^M(c): c \in Q(M) \} =   \{D^N(c): c \in Q(N) \}$.    Then for each $n$, $\ee$ refines the ydlept $\mer$ corresponding to the $0$-definable set $D(c_n)$.
  Thus the maximal CL reduct $R$ of $\ee$ contains the definable sets $D(c_n)$.  
  We expect that making $(D,Q,I)$ as generic as possible subject to the above, $R$ will be generated by the $D(c_n)$, but no
  finite set of these will generate $R$.   
  \end{rem}
 \begin{question}  Can a construction similar to \ref{question:finitelygen}, with suitable $I$, provide a $2$-yclept $\ee$ such that the  maximal 
CL reduct preserved by $\ee$ does not generate an ydlept equivalence relation?
 \end{question}
 
.

\bibliographystyle{alpha}
\bibliography{modeleq}

 \pagebreak

 \tableofcontents

 \end{document}